\documentclass[12pt,a4paper]{article}

\usepackage{latexsym}
\usepackage{amsmath}
\usepackage{amssymb}
\usepackage{amsthm}
\usepackage{amscd}
\usepackage{mathrsfs}
\usepackage[all]{xy}
\usepackage{graphicx}
\usepackage{comment}
\usepackage{arydshln}

\newtheorem{thm}{Theorem}[section]
\newtheorem{prop}[thm]{Proposition}

\newtheorem{lem}[thm]{Lemma}
\newtheorem{cor}[thm]{Corollary}

\newtheorem{rem}[thm]{Remark}

\newtheorem*{thmn}{Theorem}

\makeatletter

\@addtoreset{equation}{section}
\makeatother

\makeatletter

\@addtoreset{equation}{section}
\makeatother

\makeatletter
\newcommand{\subsubsubsection}{\@startsection{paragraph}{4}{\z@}%
 {1.0\Cvs \@plus.5\Cdp \@minus.2\Cdp}%
 {.1\Cvs \@plus.3\Cdp}%
 {\reset@font\sffamily\normalsize}
 }
\makeatother
\DeclareMathOperator{\Spec}{Spec}

\DeclareMathOperator{\Gal}{Gal}

\DeclareMathOperator{\Hom}{Hom}

\DeclareMathOperator{\Ker}{Ker}

\DeclareMathOperator{\Ind}{Ind} 
\DeclareMathOperator{\Tr}{Tr}

\DeclareMathOperator{\Irr}{Irr}

\DeclareMathOperator{\GL}{GL}

\DeclareMathOperator{\Sp}{Sp}

\DeclareMathOperator{\SU}{SU}

\DeclareMathOperator{\HU}{HU}
\DeclareMathOperator{\HSp}{HSp}
\DeclareMathOperator{\HSpU}{HSpU}

\DeclareMathOperator{\oU}{U}

\DeclareMathOperator{\oH}{H}
\DeclareMathOperator{\oO}{O}

\newcommand{\bF}{\mathbb{F}}
\newcommand{\bG}{\mathbb{G}}

\newcommand{\bP}{\mathbb{P}}
\newcommand{\bQ}{\mathbb{Q}}

\newcommand{\bZ}{\mathbb{Z}}

\newcommand{\cO}{\mathcal{O}}

\newcommand{\ol}{\overline}

\newcommand{\wt}{\widetilde}

\newcommand{\cf}{\textit{cf.\ }}

\usepackage[usenames,dvipsnames]{color}

\begin{document}

\title
{Geometric construction of Heisenberg--Weil representations for finite unitary groups\\ and Howe correspondences} 

\author{Naoki Imai and Takahiro Tsushima}
\date{}
\maketitle
\begin{abstract}
We give a geometric construction of 
the Heisenberg--Weil representation of a finite unitary group 
by the middle \'{e}tale cohomology of an algebraic variety 
over a finite field, whose rational points give 
a unitary Heisenberg group. 
Using also a Frobenius action, 
we give a geometric realization of the Howe correspondence for 
$(\mathrm{Sp}_{2n},\mathrm{O}_2^-)$ 
over any finite field including characteristic two. 
As an application, we show that unipotency is preserved 
under the Howe correspondence. 
\end{abstract}


\section{Introduction}
For a reductive dual pair $(G,G')$ 
over a field (\cf \cite[\S 5]{Howthetainv}), 
by regarding the Weil representation  
as a representation of $G \times G'$ 
and decomposing it into irreducible components, 
we have a correspondence from irreducible representations of $G$ to representations of $G'$, 
which is called the Howe correspondence for $(G,G')$.  
Let $q$ be a power of a prime number $p$. 
If $q$ is odd, 
Weil representations of symplectic groups over 
$\mathbb{F}_q$ are 
studied in \cite{SaiRepsym} and \cite{HoweCharW} 
after \cite{Weiopuni}. 
Weil representations of general linear groups over $\mathbb{F}_q$ 
and unitary groups for the extension $\mathbb{F}_{q^2}/\bF_q$ 
are constructed in \cite{GerWeil} for any $q$ 
(\cf \cite{LehWuni} in the unitary case). 

In this paper, we give a geometric construction of 
Weil representations of unitary groups. 
A finite unitary group $\oU_n(\bF_q)$ of degree $n$ is regarded as 
a subgroup of an automorphism group of a certain Heisenberg group 
$\oH(h_n)$ determined by a hermitian form $h_n$ on $\bF_{q^2}^n$, 
which we call a unitary Heisenberg group (\cf \cite[Lemma 3.3]{GerWeil} and \S \ref{unit}).  
For a non-trivial character $\psi$ of 
the center of $\oH(h_n)$, 
there exists a unique irreducible representation $\rho_{\oH (h_n),\psi}$ of 
$\oH(h_n)$ which contains $\psi$.  
In \cite[Theorem 3.3]{GerWeil}, 
an irreducible representation $\rho_{\HU (h_n),\psi}$ 
of the semidirect product 
$\oH(h_n) \rtimes \oU_n(\bF_q)$ extending $\rho_{\oH(h_n),\psi}$ 
is constructed, which we call the Heisenberg--Weil representation. 
The Weil representation is 
defined to be the restriction of $\rho_{\HU(h_n),\psi}$ 
to the unitary group. 

Our geometric construction is very simple. 
We regard the unitary Heisenberg group 
as the set of rational points of an algebraic variety, 
which has a natural action of 
$\oH(h_n) \rtimes \oU_n(\bF_q)$. 
This algebraic variety is isomorphic to the affine smooth variety 
defined by 
$a^q+a=\sum_{i=1}^n x_i^{q+1}$ in 
$\mathbb{A}_{\mathbb{F}_q}^{n+1}$, which we denote by $X_n$. 
Let $\ell \neq p$ be a prime number. 
We show the following: 
\begin{thmn}
We have an isomorphism 
\[
 H_{\rm c}^{n}(X_{n,\overline{\mathbb{F}}_q},\overline{\mathbb{Q}}_{\ell})[\psi] \simeq \rho_{\HU(h_n),\psi} 
\]
as representations of $\oH(h_n) \rtimes \oU_n(\bF_q)$. 
\end{thmn}
We give also a branching formula for the Weil representation 
of a unitary group in Proposition \ref{prop:bran}. 
Thanks to the geometric construction, 
the branching formula is a simple consequence of 
the K\"{u}nneth formula. 

By taking the mod $\ell$ cohomology, 
we can obtain a modular Heisenberg--Weil representation 
of a unitary group without ambiguity of semi-simplifications. 
Another geometric approach for the Weil representations 
of symplectic groups is given in \cite{GuHaGeoW}, 
which is quite different from ours. 

For a unitary group of even degree, 
we can give a geometric construction of 
the Weil representation 
using a rational form of the variety $X_{2n}$, 
which is denoted by $X_{2n}'$. 
Using the Frobenius action coming from the rationality of $X_{2n}'$, 
we can construct a representation of 
$\oU_{2n}(\bF_q) \rtimes \Gal (\mathbb{F}_{q^2}/\mathbb{F}_q)$ 
as the middle ${\ell}$-adic cohomology of the variety. 
We write $W_n$ for the obtained 
$\oU_{2n}(\bF_q) \rtimes 
 \Gal (\mathbb{F}_{q^2}/\mathbb{F}_q)$-representation. 
This Frobenius action is important for 
an application to the Howe correspondence 
as we explain below. 

In characteristic two, a general formulation of 
the Howe correspondence for symplectic-orthogonal case 
is not known (\cf \cite[\S 1]{TieDualcl}). 
However, there is a trial to construct such a correspondence in 
\cite[\S 7]{GuTiCross}, where they consider a dual pair 
$(\Sp_{2n},\oO_2^{\pm})$ and construct a representation 
of the product group 
$\Sp_{2n}(\bF_q) \times \oO_2^{\pm}(\bF_q)$. 
In this paper, we give a geometric construction of the Howe correspondence for 
the reductive dual pair $(\Sp_{2n},\oO_2^-)$ as follows. 
The group $\oO_2^-(\bF_q)$ is isomorphic to 
the dihedral group $D_{2(q+1)}$. Hence, we can identify  
$\oO_2^-(\bF_q)$ with 
$\oU_1(\bF_q) \rtimes \Gal (\mathbb{F}_{q^2}/\mathbb{F}_q)$. 
We have natural homomorphisms 
$\Sp_{2n}(\bF_q) \hookrightarrow \oU_{2n}(\bF_q)$ and 
$ \oU_{2n}(\bF_q) \times \oU_1(\bF_q) \to  \oU_{2n}(\bF_q)$ 
(\cf \cite[\S1]{SriWcl}). 
Hence, we have a homomorphism 
\[
 \Sp_{2n}(\bF_q ) \times \oO_2^-(\bF_q )
 \to \oU_{2n}(\bF_q ) \rtimes \Gal (\mathbb{F}_{q^2}/\mathbb{F}_q). 
\]
Inflating $W_n$ under this homomorphism, we obtain 
a representation of $\Sp_{2n}(\bF_q ) \times \oO_2^-(\bF_q )$. 
Using this representation, 
we define the Howe correspondence for 
$(\Sp_{2n},\oO_2^-)$. 
We note that the Frobenius action coming from the rationality of $X_{2n}'$ 
plays an important role in the construction of the Howe correspondence. 
Since the construction is geometric, we can relate 
the representation of 
$\Sp_{2n}(\bF_q ) \times \oO_2^-(\bF_q )$ 
with a Lusztig induction in a geometric way. 
Using the relation, we can show that 
the Howe correspondence preserves unipotency for any $q$. 
The preservation of the unipotency is proved in 
\cite[Theorem 3.5]{AdMoUnip} for symmetric and even-orthogonal pairs 
if $p$ is odd and $q$ is large enough. 

In Section \ref{sec:HW}, 
we recall the Heisenberg--Weil representation for a unitary group and 
give a geometric realization of the representation in the 
$\ell$-adic cohomology of the variety $X_n$. 
In Section \ref{sec:An}, 
we give a geometric realization the Heisenberg--Weil representation 
for a unitary group of even degree using another coordinate. 
We study also a Frobenius action, which we need later. 
In Section \ref{sec:RelFer}, 
we relate the cohomology of the variety $X_n$ 
with the cohomology of a Fermat variety 
or its complement in a projective space. 
In Section \ref{sec:Lus}, 
we recall some facts on the Lusztig induction. 
In Section \ref{sec:uni}, 
we relate the Weil representation of a unitary group 
with a Lusztig induction. 
We give also a branching formula for the Weil representation 
of a unitary group. 
In Section \ref{sec:sop}, 
we construct a representation of 
$\Sp_{2n}(\bF_q) \times \oO_2^-(\bF_q)$ and 
define the Howe correspondence for $(\Sp_{2n},\oO_2^-)$. 
We show that this Howe correspondence 
is compatible with the usual one, which is defined if $p \neq 2$, up to an explicit sign. 
We relate the representation with a Lusztig induction and 
show the preservation of unipotency under the Howe correspondence. 

In \cite{ITShinlift}, we construct 
Shintani lifts for Weil representations of unitary groups 
as an application of the geometric construction in this paper. 
In a subsequent paper \cite{ITModlW}, we study mod $\ell$ cohomology 
and a mod $\ell$ Howe correspondence 
using results in this paper. 

\subsection*{Acknowledgements}
The authors would like to thank a referee for 
helpful comments and suggestions. 
This work was supported by JSPS KAKENHI Grant Numbers 17K18722, 20K03529, 21H00973. 

\subsection*{Notation}
For a scheme $X$ over a field $F$ and a field extension $F'/F$, 
let $X_{F'}$ denote the base change of $X$ to $F'$. 

For a finite group $G$, let $1_G$ denote the trivial representation of $G$ over a field. We simply write $1$ for $1_G$ when $G$ is clear from the context. 
For a finite group $G$, an irreducible $G$-representation 
$\lambda$ and a finite-dimensional 
$G$-representation $\rho$, let $\rho[\lambda]$ denote the $\lambda$-isotypic part of $\rho$. 

\section{Heisenberg--Weil representation for unitary group}\label{sec:HW}
\subsection{Unitary Heisenberg group}\label{unit}
We recall the unitary Heisenberg group. 
Let $q$ be a power of a prime number $p$. 
Let $V$ be a finite-dimensional 
$\mathbb{F}_{q^2}$-vector space 
with a nondegenerate $\varepsilon$-hermitian form $h$, 
where $\varepsilon \in \{ \pm 1 \}$. 
We put 
\[
 \oH(V,h)=\{ (v,a) \in V \times \mathbb{F}_{q^2} \mid a +
 \varepsilon a^q =h(v,v) \} . 
\]
We regard $\oH(V,h)$ as a group with multiplication 
\begin{equation}\label{hei}
 (v,a)(v',a')=(v+v',a+a'+h(v,v')). 
\end{equation}
We put 
$\bF_{q,\varepsilon} =\{ a \in \bF_{q^2} \mid a +\varepsilon a^q =0\}$. 
We sometimes abbreviate $\pm 1$ as $\pm$. 
The center $Z$ of $\oH(V,h)$ equals 
$\{0\} \times \bF_{q,\varepsilon}$. 
The quotient $\oH(V,h)/Z$ is identified with $V$.  
Hence, $\oH(V,h)$ is a Heisenberg group. 

Let $\ell \neq p$ be a prime number. 
For each non-trivial character 
$\psi$ of $Z$ over $\overline{\mathbb{Q}}_{\ell}$, 
there exists a unique irreducible representation of 
$\oH(V,h)$ whose restriction to $Z$ 
contains $\psi$ by 
the Stone--von Neumann theorem. 
We denote by $\rho_{\oH(V,h),\psi}$ 
the unique irreducible representation of $\oH(V,h)$ containing $\psi$. 
The dimension of $\rho_{\oH(V,h),\psi}$ equals the square 
root of the index $[\oH(V,h):Z]$. 

Let $\oU (V,h)$ denote the isometry group of $(V,h)$. 
Then, $\oU (V,h)$ acts on $\oH(V,h)$ by 
$(v,a) \mapsto (gv,a)$ for 
$g \in \oU (V,h)$ and $(v,a) \in \oH(V,h)$. 
We put 
\[
 \HU(h)= \oH(V,h) \rtimes \oU (V,h). 
\]

\begin{rem}\label{rem:HH}
Assume that $h$ is hermitian. 
We take an element 
$\xi \in \mathbb{F}_{q^2}$ 
such that $\xi^{q-1}=-1$. 
Then 
\[
 \xi h \colon V \times V \to \mathbb{F}_{q^2};\ (v,v') \mapsto \xi h(v,v') 
\]
is a skew-hermitian form. 
We have an isomorphism $\oH(V,h) \xrightarrow{\sim} 
\oH(V,\xi h);\ (v,a) \mapsto (v,\xi a)$ and an identity 
$\oU (V,h) = \oU (V,\xi h)$. 
\end{rem}

\begin{lem}\label{lem:HWcha}
There is the unique irreducible representation 
$\rho_{\HU(h),\psi}$ of $\HU(h)$ 
which is characterized by  
\begin{itemize}
\item $(\rho_{\HU(h),\psi})|_{\oH(V,h)}
\simeq \rho_{\oH(V,h),\psi}$ 
as $\oH(V,h)$-representations and 
\item
$\Tr \rho_{\HU(h),\psi}(g)=(-1)^n (-q)^{N(g)}$ 
for $g \in \oU (V,h)$, where we put 
$n=\dim_{\mathbb{F}_{q^2}} V$ and 
$N(g)=\dim_{\mathbb{F}_{q^2}} \Ker (g-1)$. 
\end{itemize}
\end{lem}
\begin{proof}
This follows from \cite[Theorem 3.3 and Theorem 4.9.2]{GerWeil} 
and Remark \ref{rem:HH}. 
\end{proof}
We call 
the representation $\rho_{\HU(h),\psi}$ in Lemma \ref{lem:HWcha} 
the Heisenberg--Weil representation of 
$\HU(h)$ for $\psi$ (\cf \cite[Proposition 3.1]{LehWuni}). 
We put 
\begin{equation}\label{eq:omegadef}
\omega_{\oU (V,h)}=\rho_{\HU(h),\psi}|_{\oU (V,h)},    
\end{equation}
which is independent of $\psi$ by Lemma \ref{lem:HWcha}. 
The representation $\omega_{\oU (V,h)}$ is 
called the Weil representation of $\oU (V,h)$.

\subsection{Cohomology of a curve}\label{CCC}

For a finite abelian group $A$, we simply write 
$A^{\vee}$ for the character group 
$\mathrm{Hom} (A,\ol{\bQ}_{\ell}^{\times})$. 

Let 
$\Lambda \in \{ \overline{\mathbb{Q}}_{\ell} , 
 \overline{\mathbb{F}}_{\ell}\}$. 
For a separated and of finite type scheme $Y$ 
over $\mathbb{F}_q$ which admits a left action of a finite group $G$, let $G$ act on $H_{\rm c}^i(Y_{\overline{\mathbb{F}}_q},\Lambda )$ as $(g^\ast)^{-1}$ for $g \in G$. 
We write $\mathbb{A}^i$ for an $i$-dimensional affine 
space over $\overline{\mathbb{F}}_q$. 
We write $\mathbb{G}_{\mathrm{m}}$ for 
 $\Spec \overline{\mathbb{F}}_q[t^{\pm 1}]$. 

For $\psi \in \Hom (\mathbb{F}_{q,\varepsilon},\Lambda^{\times})$, 
let $\mathscr{L}_{\psi}$ denote the 
rank one sheaf on 
$\mathbb{A}_{\mathbb{F}_q}^1$ obtained as 
the pushforward via $\psi^{-1} \in \Hom (\mathbb{F}_{q,\varepsilon},\Lambda^{\times})$ of 
the $\mathbb{F}_{q,\varepsilon}$-torsor over 
$\mathbb{A}_{\mathbb{F}_q}^1 =\Spec \mathbb{F}_q[t]$ 
defined by 
$z^q+\varepsilon z=t$ (\cf \cite[Sommes trig., D\'{e}finition 1.7]{DelCoet}). 
For a variety $X$ over $\mathbb{F}_q$ and a regular function 
$f \colon X \to \mathbb{A}_{\mathbb{F}_q}^1$, 
let $\mathscr{L}_{\psi}(f)$ denote the pull-back
of $\mathscr{L}_{\psi}$ by $f$. 
We put 
\[
 \mu_{q+1} = \{ \zeta \in \mathbb{F}_{q^2} \mid \zeta^{q+1}=1\}. 
\]
For $\chi \in \Hom (\mu_{q+1},\Lambda^{\times})$, 
let $\mathscr{K}_{\mathbb{G}_{\mathrm{m}},\chi}$ denote the Kummer 
sheaf on 
$\mathbb{G}_{\mathrm{m},\mathbb{F}_{q^2}}$ 
obtained as the pushforward via $\chi^{-1}$ of 
the $\mu_{q+1}$-torsor over 
$\mathbb{G}_{\mathrm{m},\mathbb{F}_{q^2}}=\Spec \mathbb{F}_{q^2}[t^{\pm 1}]$ 
defined by $y^{q+1}=t$. 

Let $C$ be the affine smooth curve defined by 
$z^q+z=x^{q+1}$ over $\mathbb{F}_{q^2}$. 
The group $\mathbb{F}_{q,+}$ acts on $C$ by 
$z \mapsto z+a$ for $a \in \mathbb{F}_{q,+}$. 

Let $\psi \in \Hom (\mathbb{F}_{q,+},\Lambda^{\times}) \setminus \{1\}$ 
in the rest of this section. 
The first claim of the following lemma is a variant of 
\cite[Lemma 7.1]{ITstab3}. 
The cohomology of a variant is studied also in \cite[\S 3.3.1]{BoRoCoxmod}. 
\begin{lem}\label{tri}
Assume that $\Lambda=\ol{\bQ}_{\ell}$. 
\begin{enumerate}
\item\label{enu:H1C} 
We have 
$H_{\rm c}^i(C_{\overline{\mathbb{F}}_q},\overline{\mathbb{Q}}_{\ell})[\psi]
=0$ for $i \neq 1$ and 
an isomorphism 
\[
 H_{\rm c}^1(C_{\overline{\mathbb{F}}_q},\overline{\mathbb{Q}}_{\ell})[\psi] 
 \simeq \bigoplus_{\chi \in 
 \mu_{q+1}^{\vee} \setminus \{1\}} 
 \overline{\mathbb{Q}}_{\ell} (\chi) 
\]
as $\mu_{q+1}$-representations, 
where $\overline{\mathbb{Q}}_{\ell} (\chi)$ denotes 
$\overline{\mathbb{Q}}_{\ell}$ with action of $\mu_{q+1}$ by $\chi$. 
In particular, we have 
$\dim H_{\rm c}^1(C_{\overline{\mathbb{F}}_q},\overline{\mathbb{Q}}_{\ell})[\psi]=q$. 
\item\label{enu:HCfree}
We regard 
$\psi$ as an element of 
$\mathrm{Hom}(\mathbb{F}_{q,+},\ol{\bZ}_{\ell}^{\times})$ via the factorization through $\ol{\bZ}_{\ell}^{\times} \subset \ol{\bQ}_{\ell}^{\times}$. 
The $\ol{\bZ}_{\ell}$-module 
$H_{\rm c}^1(C_{\overline{\mathbb{F}}_q},\ol{\bZ}_{\ell})[\psi]$ 
is free and $H^i_{\rm c}(C_{\overline{\mathbb{F}}_q},\ol{\bZ}_{\ell})[\psi]=0$ 
for $i \neq 1$. 
\item\label{enu:Cmodl} 
Let $\overline{\psi}$ denote the composite of $\psi$ and the reduction map 
$\ol{\bZ}_{\ell}^{\times} \to \ol{\bF}_{\ell}^{\times}$. 
Then we have an isomorphism 
$(H_{\rm c}^i(C_{\overline{\mathbb{F}}_q},\ol{\bZ}_{\ell})[\psi]) \otimes_{\ol{\bZ}_{\ell}} \ol{\bF}_{\ell}
\xrightarrow{\sim} H_{\rm c}^i(C_{\overline{\mathbb{F}}_q},\ol{\bF}_{\ell})[\overline{\psi}]$ 
for $i \geq 0$. 
\end{enumerate}
\end{lem}
\begin{proof}
We have isomorphisms 
\begin{equation}\label{eq:CGm}
H_{\rm c}^1(C_{\overline{\mathbb{F}}_q},\overline{\mathbb{Q}}_{\ell})[\psi] \simeq 
H_{\rm c}^1(\mathbb{A}^1,\mathscr{L}_{\psi}(x^{q+1})) \simeq \bigoplus_{\chi \in \mu_{q+1}^{\vee} \setminus \{1\}}
H_{\rm c}^1(\mathbb{G}_{\mathrm{m}},\mathscr{L}_{\psi} \otimes 
\mathscr{K}_{\mathbb{G}_{\mathrm{m}},\chi}) 
\end{equation}
as in \cite[2.2]{ITlgsw1}. 
We have 
\begin{equation}\label{som}
H_{\rm c}^i(\mathbb{G}_{\mathrm{m}},\mathscr{L}_{\psi} \otimes 
\mathscr{K}_{\mathbb{G}_{\mathrm{m}},\chi})=0\ \textrm{for $i \neq 1$}, \quad 
\dim H_{\rm c}^1(\mathbb{G}_{\mathrm{m}},\mathscr{L}_{\psi} \otimes 
\mathscr{K}_{\mathbb{G}_{\mathrm{m}},\chi})=1
\end{equation}
by \cite[Sommes trig., Proposition 4.2]{DelCoet}. 
The first assertion follows from \eqref{eq:CGm} and \eqref{som}. 
The second assertion follows from a well-known fact that 
the compactly supported 
$\cO$-cohomology of a smooth affine curve over $\overline{\mathbb{F}}_q$ 
is torsion free. 
The third assertion follows from 
the second assertion. 
\end{proof}

\subsection{Geometric construction}\label{Geom}
Let $n$ be a positive integer.
We write an element $v \in \mathbb{F}_q^n$ 
as $v=(v_k)$. 
We consider the nondegenerate hermitian 
form on $\mathbb{F}_{q^2}^n$ defined by 
\begin{equation}\label{eq:hFq2}
 h_n \colon \mathbb{F}_{q^2}^n \times \mathbb{F}_{q^2}^n 
 \to \mathbb{F}_{q^2};\ 
 ((v_k),(v'_k)) \mapsto  \sum_{k=1}^n v_k^q v'_k. 
\end{equation}
We put 
\[
 \oH(h_n) = \oH(\mathbb{F}_{q^2}^n,h_n), \quad 
 \oU (h_n) = \oU (\mathbb{F}_{q^2}^n,h_n). 
\]
Let $X_n$ be the affine smooth variety over $\mathbb{F}_{q^2}$
defined by 
\[
 z^q+z=\sum_{k=1}^n x_k^{q+1}  
\]
in $\mathbb{A}_{\mathbb{F}_{q^2}}^{n+1}
=\Spec \mathbb{F}_{q^2}[x_1,\ldots,x_n,z]$. 
Let $\oH(h_n)$ act on $X_n$ by 
\[
X_n \to X_n;\ ((x_k),z) \mapsto \Bigl( (x_k+v_k),z+a+\sum_{k=1}^n v_k^q x_k \Bigr) \quad \textrm{for $((v_k),a) \in \oH(h_n)$} 
\]
and $\oU (h_n)$ act on $X_n$ by 
\[
X_n \to X_n;\ (x,z) \mapsto (gx,z) \quad \textrm{for $g \in \oU (h_n)$}, 
\]
where we regard $x=(x_k)$ as a column vector. 
The variety $X_n$ admits an action of $\HU(h_n)$. 
Let $\mathbb{F}_{q,+}$ act on $X_n$ by 
$z \mapsto z+a$ for $a \in \mathbb{F}_{q,+}$. 
The cohomology of a variant of $X_n$ is studied in \cite[Lemma 3.6]{DudDLGG}. 

We consider the morphism 
\begin{equation}\label{pi1}
\pi \colon \mathbb{A}_{\mathbb{F}_q}^n 
\to \mathbb{A}_{\mathbb{F}_q}^1;\
(x_i)_{1 \leq i \leq n} \mapsto \sum_{i=1}^n x_i^{q+1}. 
\end{equation}
Since we have a cartesian diagram
\[
\xymatrix{
X_n \ar[rr] \ar[d] & &
\mathbb{A}_{\mathbb{F}_q}^n \ar[d]^{\pi} \\
\mathbb{A}_{\mathbb{F}_q}^1 \ar[rr]^-{z \mapsto z^q +z} & & \mathbb{A}_{\mathbb{F}_q}^1 
}
\]
we obtain an isomorphism 
\begin{equation}\label{XLpsi}
 H_{\rm c}^i(X_{n,\overline{\mathbb{F}}_q},\Lambda)[\psi] 
 \simeq 
 H_{\rm c}^i(\mathbb{A}^n,\pi^\ast \mathscr{L}_{\psi}) 
\end{equation}
for $i \geq 0$ by using the proper base change theorem and taking $\psi$-parts. 

We prepare a lemma to treat the case where 
$(n,q)=(2,2)$. 
We put $\iota=
\begin{pmatrix}
0 & 1 \\
1 & 0
\end{pmatrix} \in \oU (h_2) 
$. 

\begin{lem}\label{pre}
Assume that $\Lambda =\ol{\bQ}_{\ell}$ and $q=2$. 
We have 
\[
 \Tr( \iota; 
 H_{\rm c}^2(X_{2,\overline{\mathbb{F}}_q},\overline{\mathbb{Q}}_{\ell})[\psi])
 =-2. 
\] 
\end{lem}
\begin{proof}
Let $e_1, e_2$ be a basis of 
$H_{\rm c}^1(C_{\overline{\mathbb{F}}_q},\overline{\mathbb{Q}}_{\ell})[\psi]$. 
Using the morphism 
\[
 \mathrm{pr}_i \colon \mathbb{A}_{\mathbb{F}_q}^2 
\to \mathbb{A}_{\mathbb{F}_q}^1;\
(x_1,x_2) \mapsto x_i 
\]
for $1 \leq i \leq 2$, we have an isomorphism 
\[
 \pi^* \mathscr{L}_{\psi} \simeq  \mathrm{pr}_1^* \bigl(\mathscr{L}_{\psi}(x^{q+1})\bigr) \otimes \mathrm{pr}_2^* \bigl(\mathscr{L}_{\psi}(x^{q+1})\bigr)
\]
of sheaves on $\mathbb{A}_{\mathbb{F}_q}^2$. 
Hence we have an isomorphism 
\begin{equation}\label{eq:H2XCC}
 H_{\rm c}^2(X_{2,\overline{\mathbb{F}}_q},\overline{\mathbb{Q}}_{\ell})[\psi] 
 \simeq 
 H_{\rm c}^1(C_{\overline{\mathbb{F}}_q},\overline{\mathbb{Q}}_{\ell})[\psi] 
 \otimes 
 H_{\rm c}^1(C_{\overline{\mathbb{F}}_q},\overline{\mathbb{Q}}_{\ell})[\psi] 
\end{equation}
given by 
\eqref{XLpsi}, the K\"{u}nneth formula in 
\cite[Sommes trig., (2.4.1)*]{DelCoet} and the first isomorphism in \eqref{eq:CGm}. 
Under the isomorphism \eqref{eq:H2XCC}, 
the action of $\iota$ is 
$e_i \otimes e_j \mapsto -e_j \otimes e_i$ 
for $1 \leq i,j \leq 2$, 
where the minus sign appears 
because of the anti-commutativity of cup products. 
Hence we have the claim. 
\end{proof}

\begin{thm}\label{old}
Assume that $\Lambda =\ol{\bQ}_{\ell}$. 
We have 
$H_{\rm c}^i(X_{n,\overline{\mathbb{F}}_q},\overline{\mathbb{Q}}_{\ell})[\psi]=0$ for $i \neq n$ and 
$\dim H_{\rm c}^n(X_{n,\overline{\mathbb{F}}_q},\overline{\mathbb{Q}}_{\ell})[\psi]=q^n$. 
Further, we have an isomorphism 
\[
 H_{\rm c}^{n}(X_{n,\overline{\mathbb{F}}_q},\overline{\mathbb{Q}}_{\ell})[\psi] \simeq \rho_{\HU(h_n),\psi} 
\]
as $\HU(h_n)$-representations. 
\end{thm}
\begin{proof}
The first assertion follows from \eqref{XLpsi}, the 
K\"{u}nneth formula, Lemma \ref{tri} \ref{enu:H1C} and the first isomorphism in \eqref{eq:CGm} in the same way as the proof of Lemma \ref{pre}. 
By the first assertion, 
$H_{\rm c}^{n}(X_{n,\overline{\mathbb{F}}_q},\overline{\mathbb{Q}}_{\ell})[\psi]$
is isomorphic to $\rho_{\oH(h_n),\psi}$ as $\oH(h_n)$-representations by the Stone--von Neumann theorem. 
We write $\det$ for the composite 
$\HU(h_n) \xrightarrow{\mathrm{pr}} 
\oU (h_n) \xrightarrow{\det} \mu_{q+1}$. 

Assume that $q$ is odd. 
By \cite[(1) in the proof of Theorem 3.3]{GerWeil}, 
the finite special unitary group 
$\SU_n(\bF_q)$ is perfect except for 
$(n,q)=(2,3)$. Hence, if $(n,q) \neq (2,3)$, 
any character of $\oU (h_n)$
factors through $\det$. 
Even if $(n,q)=(2,3)$, 
any character of $\oU (h_n)$
does by \cite[the table in p.28]{EnnChar}. 

Assume that $q$ is even. 
By \cite[(8) in the proof of Theorem 3.3]{GerWeil}, 
any character of $\oU (h_n)$
factors through $\det$ except for $(n,q)=(2,2)$. 
Further, assume that $(n,q)=(2,2)$. 
By \cite[(5) in the proof of Theorem 3.3]{GerWeil}, 
the unitary group $\oU (\bF_4^2,h_2)$
is isomorphic to $\mathbb{Z}/3\mathbb{Z} \times \mathfrak{S}_3$. 
Let $\mathrm{sgn} \colon \mathfrak{S}_3 \to 
\overline{\mathbb{Q}}_{\ell}^{\times}$ be the sign character. 
Then the character group of 
$\oU (\bF_4^2,h_2)$ is generated by 
$\det$ and 
\[
 \chi_2 \colon 
\mathbb{Z}/3\mathbb{Z} \times \mathfrak{S}_3 \xrightarrow{\mathrm{pr}_2}
 \mathfrak{S}_3 \xrightarrow{\mathrm{sgn}}
 \overline{\mathbb{Q}}_{\ell}^{\times} 
\]
(\cf \cite[the table in p.28]{EnnChar}). 

By Schur's lemma, there exists a character $\chi \in \mu_{q+1}^{\vee}$ 
such that 
\[
 H_{\rm c}^{n}(X_{n,\overline{\mathbb{F}}_q},\overline{\mathbb{Q}}_{\ell})[\psi] \simeq 
 \begin{cases}
 \rho_{\HU(h_n),\psi} \otimes (\chi \circ \det) & 
 \textrm{if $(n,q) \neq (2,2)$}, \\ 
 \rho_{\HU(h_n),\psi} \otimes (\chi \circ \det) \otimes \chi_2^j & 
 \textrm{for some $0 \leq j \leq 1$ if $(n,q) = (2,2)$}. 
 \end{cases}
\]
Let $g=\mathrm{diag}(\zeta_1,\ldots,\zeta_n) \in \oU (h_n)$
with $\zeta_i \in \mu_{q+1} \setminus \{1\}$.  
By the K\"{u}nneth formula and Lemma \ref{tri} \ref{enu:H1C}, we have 
\begin{equation}\label{det}
\Tr (g; H_{\rm c}^{n}(X_{n,\overline{\mathbb{F}}_q},\overline{\mathbb{Q}}_{\ell})[\psi])=(-1)^n. 
\end{equation}
By Lemma \ref{lem:HWcha} and \eqref{det}, we have 
\begin{align*}
(-1)^n\chi(\zeta_1 \cdots \zeta_n)=
\Tr \rho_{\HU(h_n),\psi}(g) \chi(\det(g))=\Tr (g; H_{\rm c}^{n}(X_{n,\overline{\mathbb{F}}_q},\overline{\mathbb{Q}}_{\ell})[\psi])
=(-1)^n. 
\end{align*}
Hence, we have 
$\chi(\zeta_1 \cdots \zeta_n)=1$ for 
any $\zeta_i \in \mu_{q+1} \setminus \{1\}$. 
This implies $\chi=1$. 
If $(n,q) = (2,2)$, 
by Lemma \ref{lem:HWcha} and Lemma \ref{pre}, we have 
\[
 -2 (-1)^j
 = \Tr \rho_{\HU(h_n),\psi}(\iota) \chi(\det(\iota)) \chi_2(\iota)^j 
 = \Tr (\iota; H_{\rm c}^{n}(X_{n,\overline{\mathbb{F}}_q},\overline{\mathbb{Q}}_{\ell})[\psi]) 
 = -2, 
\]
which implies $j=0$. 
\end{proof}

We give a relation with mod $\ell$-cohomology. 

\begin{prop}\label{prop:Xmodl}
Let the notations be as in Lemma \ref{tri}. 
\begin{enumerate}
\item\label{enu:HXfree} 
The $\ol{\bZ}_{\ell}$-module 
$H_{\rm c}^n(X_{n,\overline{\mathbb{F}}_q},\ol{\bZ}_{\ell})[\psi]$ is free and 
$H_{\rm c}^i(X_{n,\overline{\mathbb{F}}_q},\ol{\bZ}_{\ell})[\psi]=0$ for 
$i \neq n$. 
\item\label{enu:Xmodl}
We have an isomorphism 
\[
 (H_{\rm c}^i(X_{n,\overline{\mathbb{F}}_q},\ol{\bZ}_{\ell})[\psi]) \otimes_{\ol{\bZ}_{\ell}} \ol{\bF}_{\ell} \xrightarrow{\sim}
H_{\rm c}^i(X_{n,\overline{\mathbb{F}}_q},\ol{\bF}_{\ell})[\overline{\psi}]
\] 
as $\ol{\bF}_{\ell}[\HU(h_n)]$-modules 
for $i \geq 0$. 
\end{enumerate}
\end{prop}
\begin{proof}
These assertions follow 
from the K\"{u}nneth formula, and 
Lemma \ref{tri} \ref{enu:HCfree} and \ref{enu:Cmodl}, respectively. 
\end{proof}

The $\ol{\bZ}_{\ell}[\HU(h_n)]$-module 
$H_{\rm c}^n(X_{n,\overline{\mathbb{F}}_q},\ol{\bZ}_{\ell})[\psi]$ is a 
$\ol{\bZ}_{\ell}[\HU(h_n)]$-lattice in 
$H_{\rm c}^n(X_{n,\overline{\mathbb{F}}_q},\overline{\mathbb{Q}}_{\ell})[\psi]$ by Proposition \ref{prop:Xmodl} \ref{enu:HXfree}. 
By Proposition \ref{prop:Xmodl} \ref{enu:Xmodl}, 
the $\ol{\bF}_{\ell}[\HU(h_n)]$-module 
$H_{\rm c}^n(X_{n,\overline{\mathbb{F}}_q},\ol{\bF}_{\ell})[\overline{\psi}]$
is regarded as a mod $\ell$ version of 
a Heisenberg--Weil representation of a unitary group. 

\section{Another coordinate}\label{sec:An}
We give a construction of the Heisenberg--Weil representation of a finite unitary group of even degree using another coordinate. 
Let $\psi \in \mathbb{F}_q^{\vee} \setminus \{1\}$ in this section. 

\subsection{Cohomology of a surface}\label{srf}
Let $X$ be the affine smooth surface over 
$\mathbb{F}_q$ defined by
\begin{equation}\label{xd}
 z^q-z=x y^q -x^q y 
\end{equation}
in $\mathbb{A}_{\mathbb{F}_q}^3=\Spec \mathbb{F}_q[x,y,z]$. 
Let 
$\mathbb{F}_{q^2}^{\times}$ act on $X_{\overline{\bF}_q}$ by 
$(x,y,z) \mapsto (\zeta^{-1} x,\zeta^{q} y ,z)$ for $\zeta \in \mathbb{F}_{q^2}^{\times}$. 
Let $\mathrm{Fr}_q \in \Gal (\overline{\mathbb{F}}_q/\mathbb{F}_q)$ be the geometric 
Frobenius automorphism defined by $x \mapsto x^{q^{-1}}$ for $x \in \overline{\mathbb{F}}_q$. 
When we consider a closed subscheme of a variety, we suppose that it 
is equipped with the reduced scheme structure. 

Let $\mathbb{F}_{q^2}^{\times} \rtimes \mathrm{Fr}_q^{\mathbb{Z}}$ 
be the semidirect product under the natural action. 
We sometimes regard $(\mathbb{F}_q^{\times})^{\vee}$ as a subset of 
$(\mathbb{F}_{q^2}^{\times})^{\vee}$ via the norm map. 
For $\chi \in (\mathbb{F}_{q}^{\times})^{\vee} \subset (\mathbb{F}_{q^2}^{\times})^{\vee}$, 
we define 
$\wt{\chi} \in (\mathbb{F}_{q^2}^{\times} \rtimes 
 \mathrm{Fr}_q^{\mathbb{Z}})^{\vee}$ by 
$\wt{\chi}(\zeta,\mathrm{Fr}_q^m)=\chi (\zeta)$ for 
$\zeta \in \mathbb{F}_{q^2}^{\times}$ and $m \in \bZ$. 
For $\chi \in (\mathbb{F}_{q^2}^{\times})^{\vee} \setminus (\mathbb{F}_q^{\times})^{\vee}$, 
we define 
$\wt{\chi} \in (\mathbb{F}_{q^2}^{\times} \rtimes 
 \mathrm{Fr}_q^{2\mathbb{Z}})^{\vee}$ by 
$\wt{\chi}(\zeta,\mathrm{Fr}_q^{2m})=\chi (\zeta)$ for 
$\zeta \in \mathbb{F}_{q^2}^{\times}$ and $m \in \bZ$, 
and put 
\[
 \pi_{\chi} =\Ind_{\mathbb{F}_{q^2}^{\times} \rtimes 
 \mathrm{Fr}_q^{2\mathbb{Z}}}^{\mathbb{F}_{q^2}^{\times} \rtimes 
 \mathrm{Fr}_q^{\mathbb{Z}}} \wt{\chi}. 
\]
For $\chi , \chi' \in (\mathbb{F}_{q^2}^{\times})^{\vee} \setminus (\mathbb{F}_q^{\times})^{\vee}$, we have 
$\pi_{\chi} \simeq \pi_{\chi'}$ 
if and only if 
$\chi$ and $\chi'$ are equal in 
$( (\mathbb{F}_{q^2}^{\times})^{\vee} \setminus 
 (\mathbb{F}_q^{\times})^{\vee} ) / \mathrm{Fr}_q^{\mathbb{Z}}$. 

\begin{lem}\label{ke}
We have an isomorphism 
\[
 H_{\rm c}^2(X_{\overline{\mathbb{F}}_q},\overline{\mathbb{Q}}_{\ell})(1)[\psi]
 \simeq 
 1 \oplus \bigoplus_{\chi \in (\mathbb{F}_q^{\times})^{\vee}}
 \wt{\chi} \oplus 
 \bigoplus_{\chi \in ( (\mathbb{F}_{q^2}^{\times})^{\vee} 
 \setminus (\mathbb{F}_q^{\times})^{\vee} ) / \mathrm{Fr}_q^{\mathbb{Z}}} 
 \pi_{\chi} 
\]
as $\mathbb{F}_{q^2}^{\times} \rtimes 
 \mathrm{Fr}_q^{\mathbb{Z}}$-representations. 
\end{lem}
\begin{proof}
By changing a variable as $w=z+xy^q$ in \eqref{xd}, 
the surface 
$X$ is defined by $w^q-w=x^q(y^{q^2}-y)$. 
Let $S$ be the surface defined by $w^q-w=x(y^{q^2}-y)$. 
We have the finite purely inseparable map 
$X \to S;\ (w,x,y) \mapsto (w,x^q,y)$. 
Hence we have an isomorphism 
$H_{\rm c}^2(X_{\overline{\mathbb{F}}_q},\overline{\mathbb{Q}}_{\ell}) \simeq 
 H_{\rm c}^2(S_{\overline{\mathbb{F}}_q},\overline{\mathbb{Q}}_{\ell})$. 

Let $\pi \colon \mathbb{A}_{\mathbb{F}_{q^2}}^2 \to 
\mathbb{A}_{\mathbb{F}_{q^2}}^2;\ (x,y) \mapsto
(x,y^{q^2}-y)$. 
For any $\psi' \in \mathbb{F}_{q^2}^{\vee}$, there 
exists a unique element $\alpha \in \mathbb{F}_{q^2}$ such that $\psi'(x)=
\psi(\Tr_{\mathbb{F}_{q^2}/\mathbb{F}_q}(\alpha x))$ for any $x \in \mathbb{F}_{q^2}$. 
We have isomorphisms 
\begin{align*}
 H_{\rm c}^2(S_{\overline{\mathbb{F}}_q},\overline{\mathbb{Q}}_{\ell})[\psi]  
 & \simeq 
 H_{\rm c}^2(\mathbb{A}^2,\mathscr{L}_{\psi}(x(y^{q^2}-y)) 
 =H_{\rm c}^2(\mathbb{A}^2,\pi^\ast \mathscr{L}_{\psi}(xy)) \\ 
 & \simeq 
 H_{\rm c}^2(\mathbb{A}^2,\pi_{\ast} \pi^\ast \mathscr{L}_{\psi}(xy)) 
 \simeq 
 H_{\rm c}^2(\mathbb{A}^2,\mathscr{L}_{\psi}(xy) \otimes \pi_{\ast} 
 \overline{\mathbb{Q}}_{\ell}) \\
 & \simeq \bigoplus_{\psi' \in \mathbb{F}_{q^2}^{\vee}} 
 H_{\rm c}^2(\mathbb{A}^2, \mathscr{L}_{\psi}(xy) \otimes 
 \mathscr{L}_{\psi'}(y))
 \simeq 
 \bigoplus_{\alpha \in \mathbb{F}_{q^2}}
 H_{\rm c}^2(\mathbb{A}^2, \mathscr{L}_{\psi}((x+\alpha)y)), 
\end{align*}
where we use the projection formula at the fourth isomorphism. 
Let $\mathscr{F}_{\psi}$ denote the $\ell$-adic 
Fourier transformation on $\mathbb{A}^1$ associated to $\psi$ in 
\cite[D\'{e}finition (1.2.1.1)]{LauTFcW}.  
By \cite[Proposition (1.2.2.2)]{LauTFcW}, we have 
\[
 H_{\rm c}^2(\mathbb{A}^2, \mathscr{L}_{\psi}((x+\alpha)y)) \simeq H_{\rm c}^2(\mathbb{A}^2, \mathscr{L}_{\psi}(xy)) \simeq  H_{\rm c}^1(\mathbb{A}^1, \mathscr{F}_{\psi}(\overline{\mathbb{Q}}_{\ell}))
\simeq \overline{\mathbb{Q}}_{\ell}(-1). 
\]
Hence, we obtain the claim. 
\end{proof}

\subsection{Construction using another coordinate}\label{another}
Let $n$ be a positive integer. 
Let $V=\mathbb{F}_q^{2n}$. 
We consider the skew-hermitian 
form on 
$V_{\mathbb{F}_{q^2}}=\mathbb{F}_{q^2}^{2n}$ defined by 
\begin{equation}\label{df}
h_{2n}' \colon V_{\mathbb{F}_{q^2}} 
\times V_{\mathbb{F}_{q^2}}
\to \mathbb{F}_{q^2};\ 
((v_k),(v'_k)) \mapsto 
\sum_{k=1}^n\left( v_k^q v'_{n+k} -v_{n+k}^q v'_k  \right). 
\end{equation}
Let $X_{2n}'$ be the affine smooth variety over $\mathbb{F}_q$ defined by 
\[
 z- z^q =\sum_{k=1}^n (x_k^q x_{n+k} - x_{n+k}^q x_k ) 
\] 
in 
$\mathbb{A}_{\mathbb{F}_q}^{2n+1}=\Spec \mathbb{F}_q[x_1,\ldots,x_{2n},z]$. 
The group $\oH(V_{\mathbb{F}_{q^2}},h_{2n}')$ acts on $X_{2n}'$ 
by 
\[
X_{2n}' \to X_{2n}';\ ((x_k),z) \mapsto 
 \biggl( (x_k +v_k),z+a+\sum_{k=1}^n ( v_k^q x_{n+k} - v_{n+k}^q x_k ) \biggr) 
\]
for $(v,a) \in \oH(V_{\mathbb{F}_{q^2}},h_{2n}')$. 
The group $\oU (V_{\mathbb{F}_{q^2}},h_{2n}')$ acts on $X_{2n}'$ by $(x,z) \mapsto (g x ,z)$ for $g \in \oU (V_{\mathbb{F}_{q^2}},h_{2n}')$, 
where we regard $x=(x_k)$ as a column vector.  
Hence, $\HU(h_{2n}')$
acts on $X_{2n}'$. 
Let $\mathbb{F}_q$ act on $X_{2n}'$ by 
$z \mapsto z+a$ for $a \in \mathbb{F}_q$. 

\begin{rem}\label{rem:form}
Let $h_{2n}$ be the hermitian form on $V_{\mathbb{F}_{q^2}}=\bF_{q^2}^{2n}$ 
defined in \eqref{eq:hFq2}. 
Take $\xi \in \mathbb{F}_{q^2}$ 
such that $\xi^{q-1}=-1$. 
Then we have an isomorphism 
$f \colon (V_{\mathbb{F}_{q^2}},\xi h_{2n}) \to 
 (V_{\mathbb{F}_{q^2}},h_{2n}')$ as skew-hermitian forms. 
This induces isomorphisms 
$\oH(V_{\mathbb{F}_{q^2}},\xi h_{2n}) \simeq \oH(V_{\mathbb{F}_{q^2}},h_{2n}')$ 
and 
$\oU (V_{\mathbb{F}_{q^2}},\xi h_{2n}) 
 \simeq \oU (V_{\mathbb{F}_{q^2}},h_{2n}')$ 
by Remark \ref{rem:HH}. 
Further, $f$ and $z \mapsto \xi z$ gives an 
isomorphism $X_{2n} \simeq X_{2n,\mathbb{F}_{q^2}}'$ 
over $\mathbb{F}_{q^2}$, 
which is compatible with group actions 
under the above isomorphisms. 
\end{rem}

\begin{lem}\label{lem:F22}
We have 
$H_{\rm c}^i(X_{2n,\overline{\mathbb{F}}_q}',\overline{\mathbb{Q}}_{\ell})[\psi]=0$ for 
$i \neq 2n$ and 
$\dim H_{\rm c}^{2n}(X_{2n,\overline{\mathbb{F}}_q}',\overline{\mathbb{Q}}_{\ell})[\psi]=q^{2n}$. 
Further, 
we have an isomorphism 
\[
 H_{\rm c}^{2n}(X_{2n,\overline{\mathbb{F}}_q}',
 \overline{\mathbb{Q}}_{\ell})[\psi] \simeq \rho_{\HU(h_{2n}'),\psi} 
\]
as $\HU(h_{2n}')$-representations. 
\end{lem}
\begin{proof}
These follow from 
Theorem \ref{old} and 
Remark \ref{rem:form}. 
\end{proof}

\begin{lem}\label{lem:F2triv}
The action of $\mathrm{Fr}_q^2$ on 
$H_{\rm c}^{2n}(X_{2n,\overline{\mathbb{F}}_q}',\overline{\mathbb{Q}}_{\ell}(n))[\psi]$ is trivial. 
\end{lem}
\begin{proof}
This follows from Lemma \ref{ke} and the K\"unneth formula. 
\end{proof}

\section{Relation with Fermat variety}\label{sec:RelFer}

\subsection{Unipotent representations of unitary group}\label{reviewu}
Let $n$ be a positive integer. 
We follow \cite[\S1]{HoMaTT}. 
Let $G_n$ denote a general linear group 
$\GL_n$ over $\overline{\mathbb{F}}_q$. 
We consider the Frobenius endomorphism 
$F \colon G_n \to G_n;\ 
(x_{i,j}) \mapsto (x_{j,i}^q)^{-1}$. 
Let $\oU_n$ be the unitary algebraic group over $\bF_q$ 
defined by $h_n$. 
Then we have $G_n^F=\oU_n(\bF_q)=\oU (h_n)$. 

Let $T_0$ denote the $F$-stable maximal torus 
of $G_n$
consisting of diagonal matrices. We set 
$\mathcal{W}_n=N_{G_n}(T_0)/T_0$. Let $\mathscr{T}_n$ denote the set of $G_n^F$-conjugacy classes of $F$-stable maximal tori. 
The conjugacy classes of $\mathcal{W}_n$ is identified with 
$\mathscr{T}_n$ by 
$x^{-1}F(x) \mapsto x T_0 x^{-1}$ 
for $x^{-1}F(x) \in \mathcal{W}_n$. 
The Weyl group $\mathcal{W}_n$ is isomorphic to 
the symmetric group $\mathfrak{S}_n$. 
The set of conjugacy classes of $\mathfrak{S}_n$ 
is identified with the set of partitions of $n$, which we denote by $\Lambda_n$. 
We have the natural bijection 
\[
\Lambda_n \xrightarrow{\sim} \mathscr{T}_n;\ \rho \mapsto T(\rho) 
\]
such that $T((1^n))=T_0$. 

Let $T \in \mathscr{T}_n$ and 
$\theta \in (T^F)^{\vee}$. 
Let $R_T^{G_n}(\theta)$ denote the Deligne--Lusztig character in the notation of \cite[p.204]{LusFinunip} and \cite[Corollary 2.4]{LusRepChe}. 
Let $\chi^{\lambda}$ denote the irreducible character of $\mathfrak{S}_n$ corresponding to 
$\lambda \in \Lambda_n$ normalized such that 
$\chi^{\lambda}$ is the sign representation if $\lambda=(1^n)$. 
For $\lambda, \rho \in \Lambda_n$, 
let $\chi_{\rho}^{\lambda}$
denote the value of $\chi^{\lambda}$ at the class 
corresponding to $\rho$. 
Let $z_{\rho}$ be the cardinality of the centralizer of 
the class in $\mathfrak{S}_n$ corresponding to $\rho$. 
For $\lambda \in \Lambda_n$, we define a class function $\psi^{\lambda}$ on $G_n^F$ by 
\begin{equation}\label{class}
\psi^{\lambda}=\sum_{\rho \in \Lambda_n} \frac{\chi_{\rho}^{\lambda}}{z_{\rho}} R_{T(\rho)}^{G_n}(1_{T(\rho)^F}). 
\end{equation}
We follow the definition of unipotency in \cite[Definition 7.8]{DeLuRep}. 
By \cite[\S 2]{LuSrChar}, 
the set $\{\psi^{\lambda}\}_{\lambda \in \Lambda_n}$
equals the set of all unipotent characters of $G_n^F$ up to sign.

\subsection{Geometric relation}\label{ssec:Grel}
Let $\Lambda \in \{ \ol{\bQ}_{\ell}, \ol{\bF}_{\ell} \}$ 
and $\psi \in \Hom (\bF_{q,+}, \Lambda^{\times}) \setminus \{1\}$. 
Let $\pi$ be as in \eqref{pi1}. 
We put $U=\pi^{-1}(\mathbb{G}_{\mathrm{m},\mathbb{F}_q})$
and $Z=\pi^{-1}(0)$. 
Then 
$U_{\mathbb{F}_{q^2}}$ and $Z_{\bF_{q^2}}$ admit 
left actions induced by the natural action of 
$\oU_n (\bF_q)$ on 
$\mathbb{A}_{\mathbb{F}_{q^2}}^n$. 
Let $\mathscr{L}_{\psi}$ be as in \S \ref{CCC}. 
In this section, all maps between 
$\oU_n (\bF_q)$-representations are 
$\oU_n (\bF_q)$-equivariant. 

\begin{lem}\label{iso2}
We have a long exact sequence 
\begin{align*}
0 &\to H_{\rm c}^{n-1}(Z_{\overline{\mathbb{F}}_q},\Lambda) \to 
H_{\rm c}^n(U_{\overline{\mathbb{F}}_q},\pi^\ast \mathscr{L}_{\psi}) \to 
H_{\rm c}^n(\mathbb{A}^n,\pi^\ast \mathscr{L}_{\psi}) \to 
H_{\rm c}^n(Z_{\overline{\mathbb{F}}_q},\Lambda) \\
 & \xrightarrow{\delta} H_{\rm c}^{n+1}(U_{\overline{\mathbb{F}}_q},\pi^\ast \mathscr{L}_{\psi})  \to 0. 
\end{align*}
\end{lem}
\begin{proof}
We have $H_{\rm c}^i(\mathbb{A}^n,\pi^\ast \mathscr{L}_{\psi})=0$ for $i\neq n$ by Theorem \ref{old} and the K\"{u}nneth formula. 
We have $(\pi^\ast \mathscr{L}_{\psi})|_Z=\Lambda$. Hence the assertion follows. 
\end{proof}

\begin{lem}\label{special}
Assume that $\Lambda =\ol{\bQ}_{\ell}$. 
Let $\chi \in \mu_{q+1}^{\vee}$. 
Then $H_{\rm c}^n(\mathbb{A}^n,\pi^\ast \mathscr{L}_{\psi})[\chi]$ 
is an irreducible $\oU_n (\bF_q)$-representation, and 
we have 
\begin{align*}
\dim H_{\rm c}^n (\mathbb{A}^n,\pi^\ast \mathscr{L}_{\psi})[\chi] 
=\begin{cases}
\displaystyle \frac{q^n+(-1)^n q}{q+1} & \textrm{if $\chi=1$}, \\
\displaystyle \frac{q^n-(-1)^n}{q+1} & \textrm{if $\chi \neq 1$}. 
\end{cases}
\end{align*}
\end{lem}
\begin{proof}
The assertion follows from Theorem \ref{old} and 
\cite[Corollary 4.5(a) and its proof]{GerWeil}. 
\end{proof}

Let $S_n$ be the Fermat variety defined by 
$\sum_{i=1}^n x_i^{q+1}=0$ in $\mathbb{P}_{\mathbb{F}_q}^{n-1}$. 
Let $Y_n=\mathbb{P}_{\mathbb{F}_q}^{n-1} \setminus S_n$.  
Let $\oU_n (\bF_q)$ act on $\mathbb{P}_{\mathbb{F}_{q^2}}^{n-1}$ by usual left multiplication. 
Then $S_{n,\mathbb{F}_{q^2}}$ is stable under the action. 

Let $\widetilde{Y}_n$ be the affine smooth variety 
defined by $\sum_{i=1}^n y_i^{q+1}=1$ in 
$\mathbb{A}_{\mathbb{F}_q}^n$. 
Let 
\[
f_{Y_n} \colon \widetilde{Y}_n \to 
Y_n;\ (y_i)_{1 \leq i \leq n} \mapsto [y_1:\cdots:y_n]. 
\]
Let $\mu_{q+1}$ act on $Y_{n,\bF_{q^2}}$ by $(y_i)_{1 \leq i \leq n} \mapsto (\zeta y_i)_{1 \leq i \leq n}$ for $\zeta \in \mu_{q+1}$. Then $Y_{n,\bF_{q^2}}$ is a $\mu_{q+1}$-torsor over $Y_{n,\mathbb{F}_{q^2}}$. 
For $\chi \in \mu_{q+1}^{\vee}$, let 
$\mathscr{K}_{Y_n,\chi}$ denote the smooth 
$\Lambda$-sheaf on $Y_{n,\mathbb{F}_{q^2}}$ associated to 
$f_{Y_n}$ and $\chi^{-1}$. 
We note that 
$\mathscr{K}_{Y_n,\chi}$ is defined over 
$Y_n$ if $\chi^2 =1$. 

In the sequel, we assume that $q+1$ is invertible in $\Lambda$. 

\begin{lem}\label{fol}
We have an isomorphism
$H_{\rm c}^i(U_{\overline{\mathbb{F}}_q},
\pi^\ast \mathscr{L}_{\psi})[\chi] \simeq 
H_{\rm c}^{i-1}(Y_{n,\overline{\mathbb{F}}_q},\mathscr{K}_{Y_n,\chi})$ 
as $\oU_n(\bF_q)$-representations 
for $\chi \in \Hom (\mu_{q+1},\Lambda^{\times})$ 
and any integer $i$. 
If $\Lambda=\ol{\bQ}_{\ell}$ and $\chi^2 =1$, 
then we have an isomorphism 
\[
 H_{\rm c}^i(U_{\overline{\mathbb{F}}_q},
 \pi^\ast \mathscr{L}_{\psi})[\chi] \simeq 
 H_{\rm c}^{i-1}(Y_{n,\overline{\mathbb{F}}_q},\mathscr{K}_{Y_n,\chi}) 
 \otimes \delta_{\chi',\psi} 
\]
as representations of $\oU_n(\bF_q)$ and $\Gal (\ol{\bF}_q/\bF_q)$, 
where 
$\chi'$ is the character of $\bF_q^{\times}$ of the same order as 
$\chi$ and 
$\delta_{\chi',\psi}$ denotes the unramified character of 
$\Gal (\ol{\bF}_q/\bF_q)$ such that 
the $q$-th power geometric Frobenius element acts by 
the Gauss sum $\tau(\chi',\psi)$ defined in 
\cite[Sommes trig., (4.1.1)]{DelCoet}. 
\end{lem}
\begin{proof}
We consider the affine smooth variety 
\[
\widetilde{U}=\Biggl\{ ((x_i)_{1 \leq i \leq n},t) \in \mathbb{A}_{\mathbb{F}_q}^n \times \mathbb{G}_{\mathrm{m},\mathbb{F}_q} \Biggm| 
 \sum_{i=1}^n x_i^{q+1}=t^{q+1} \Biggr\} 
\]
over $\mathbb{F}_q$. 
The space $\widetilde{U}_{\mathbb{F}_{q^2}}$ admits 
a left action induced by the natural action of 
$\oU_n (\bF_q)$ on 
$\mathbb{A}_{\mathbb{F}_{q^2}}^n$. 
Let $\mu_{q+1}^2$ act on $\widetilde{U}_{\mathbb{F}_{q^2}}$ by
\[
 \widetilde{U}_{\mathbb{F}_{q^2}} \to \widetilde{U}_{\mathbb{F}_{q^2}};\ 
 ((x_i)_{1 \leq i \leq n},t) \mapsto 
 ((\zeta_1 x_i)_{1 \leq i \leq n},\zeta_2 t)
\]
for $(\zeta_1,\zeta_2) \in \mu_{q+1}^2$. 
Clearly, the action of $\mu_{q+1}^2$ on 
$\widetilde{U}_{\mathbb{F}_{q^2}}$ is free.  
We have a natural identification 
\[
 U=\Biggl\{ ((x_i)_{1 \leq i \leq n},s) \in 
 \mathbb{A}_{\mathbb{F}_q}^n \times \mathbb{G}_{\mathrm{m},\mathbb{F}_q} 
 \Biggm| \sum_{i=1}^n x_i^{q+1}=s \Biggr\}. 
\]
We have the morphism 
\begin{align*}
f_U & \colon \widetilde{U} \to U;\ 
((x_i)_{1 \leq i \leq n},t) \mapsto 
\left((x_i)_{1 \leq i \leq n},t^{q+1}\right). 
\end{align*} 
We put $H = \{ (1,\zeta) \in \mu_{q+1}^2 \mid \zeta \in \mu_{q+1} \}$. 
Then 
$f_U$ is an $H$-torsor over $\mathbb{F}_{q^2}$. 

Let $f_{\bG_{\mathrm{m}}} \colon 
\mathbb{G}_{\mathrm{m},\mathbb{F}_q} \to 
\mathbb{G}_{\mathrm{m},\mathbb{F}_q};\ t \mapsto t^{q+1}$. 
We have the isomorphism 
\[
 \widetilde{\varphi} \colon \widetilde{U} \xrightarrow{\sim} 
 \widetilde{Y}_n \times \mathbb{G}_{\mathrm{m},\mathbb{F}_q};\
 ( (x_i)_{1 \leq i \leq n},t) \mapsto ((x_i/t)_{1 \leq i \leq n},t).  
\]
Let $\mu_{q+1}^2$ act on $\widetilde{Y}_{n,\mathbb{F}_{q^2}} \times \mathbb{G}_{\mathrm{m},\mathbb{F}_{q^2}}$ by 
\begin{equation}\label{fog}
 \widetilde{Y}_{n,\mathbb{F}_{q^2}} \times \mathbb{G}_{\mathrm{m},\mathbb{F}_{q^2}}
\to 
\widetilde{Y}_{n,\mathbb{F}_{q^2}} \times \mathbb{G}_{\mathrm{m},\mathbb{F}_{q^2}};\ 
((y_i)_{1 \leq i \leq n},t) \mapsto ((\zeta_1y_i)_{1 \leq i \leq n},\zeta_2t)
\end{equation}
for $(\zeta_1,\zeta_2) \in \mu_{q+1}^2$. 
The morphism $\widetilde{\varphi}$ is compatible with 
$\phi \colon \mu_{q+1}^2 \xrightarrow{\sim} \mu_{q+1}^2 ;\ (\zeta_1,\zeta_2) \mapsto (\zeta_1/\zeta_2,\zeta_2)$
in the sense that $\widetilde{\varphi}(g x)=\phi(g) \widetilde{\varphi}(x)$
for $x \in \widetilde{U}_{\mathbb{F}_{q^2}}$ and $g \in \mu_{q+1}^2$. 
The isomorphism $\widetilde{\varphi}$ induces an isomorphism 
\[
 \varphi \colon U_{\mathbb{F}_{q^2}}=\widetilde{U}_{\mathbb{F}_{q^2}}/H 
 \xrightarrow{\sim} 
 (\widetilde{Y}_{n,_{\mathbb{F}_{q^2}}} \times 
 \mathbb{G}_{\mathrm{m},\mathbb{F}_{q^2}})/\phi(H) 
\]
compatible with the isomorphism 
$\mu_{q+1}^2/H \xrightarrow{\sim} \mu_{q+1}^2/\phi(H)$ induced by $\phi$. 
By taking the quotients of the both sides of $\varphi$ by 
$\mu_{q+1}^2/H$ and $\mu_{q+1}^2/\phi(H)$, 
we obtain an isomorphism 
$\overline{\varphi} \colon U_{\mathbb{F}_{q^2}}/\mu_{q+1}
\xrightarrow{\sim} Y_{n,\mathbb{F}_{q^2}} \times 
\mathbb{G}_{\mathrm{m},\mathbb{F}_{q^2}}$. 
Let 
\[
 \mathrm{pr}_1 \colon Y_{n,\mathbb{F}_{q^2}} \times \mathbb{G}_{\mathrm{m},\mathbb{F}_{q^2}} \to Y_{n,\mathbb{F}_{q^2}}, \quad 
 \mathrm{pr}_2 \colon Y_{n,\mathbb{F}_{q^2}} \times \mathbb{G}_{\mathrm{m},\mathbb{F}_{q^2}} \to \mathbb{G}_{\mathrm{m},\mathbb{F}_{q^2}}
\]
be the first projection and the second projection respectively. 
We have the commutative diagram
\[
\xymatrix{
\widetilde{U}_{\mathbb{F}_{q^2}} \ar[d]_{f_U}\ar[r]^-{\widetilde{\varphi}} & \widetilde{Y}_{n,\mathbb{F}_{q^2}} \times \mathbb{G}_{\mathrm{m},\mathbb{F}_{q^2}} \ar[d]^{\rm can.}
\ar@/^84pt/[dd]^{f_{Y_n} \times f_{\bG_{\mathrm{m}}}}\\
U_{\mathbb{F}_{q^2}} \ar[d]_{\rm can.}\ar[r]^-{\varphi} \ar@/_48pt/[dd]_{\pi} & (\widetilde{Y}_{n,\mathbb{F}_{q^2}} \times \mathbb{G}_{\mathrm{m},\mathbb{F}_{q^2}})/\phi(H) \ar[d]^{f'}\\
U_{\mathbb{F}_{q^2}}/\mu_{q+1} \ar[r]^-{\overline{\varphi}} \ar[d]_{\mathrm{pr}_2 \circ \overline{\varphi}} & Y_{n,\mathbb{F}_{q^2}} \times \mathbb{G}_{\mathrm{m},\mathbb{F}_{q^2}} \ar[dl]^{\rm pr_2}\\
\mathbb{G}_{\mathrm{m},\mathbb{F}_{q^2}},  &
}
\]
where $f'$ is a morphism induced by 
$f_{Y_n} \times f_{\bG_{\mathrm{m}}}$. 

For smooth $\Lambda$-sheaves 
 $\mathscr{F}$ on $Y_{n,\mathbb{F}_{q^2}}$ and $\mathscr{G}$ on $\mathbb{G}_{\mathrm{m},\mathbb{F}_{q^2}}$, 
let $\mathscr{F} \boxtimes \mathscr{G}$ denote the sheaf $\mathrm{pr}_1^\ast  \mathscr{F} \otimes \mathrm{pr}_2^\ast \mathscr{G}$
 on $Y_{n,\mathbb{F}_{q^2}} \times \mathbb{G}_{\mathrm{m},\mathbb{F}_{q^2}}$.  
 Let $\mathscr{K}_{\mathbb{G}_{\mathrm{m}},\chi}$ be as in \S \ref{CCC}. 
We identify as $\mu_{q+1} \xrightarrow{\sim}\mu_{q+1}^2/\phi(H);\ \zeta \mapsto (\zeta,1)$. 
We have 
 \[
 f'_{\ast} \Lambda =\left((f_{Y_n} \times f_{\bG_{\mathrm{m}}})_{\ast} \Lambda \right)^{\phi(H)} \simeq 
 \bigoplus_{\chi \in \mu_{q+1}^{\vee}}  \mathscr{K}_{Y_n,\chi} \boxtimes 
 \mathscr{K}_{\mathbb{G}_{\mathrm{m}},\chi^{-1}} 
 \]
by the K\"{u}nneth formula. 
Let $\mathscr{K}'_{\chi}$ denote the smooth 
$\Lambda$-sheaf on $U_{\mathbb{F}_{q^2}}/\mu_{q+1}$
associated to $\chi^{-1}$ and the $\mu_{q+1}$-torsor 
$U_{\mathbb{F}_{q^2}} \to U_{\mathbb{F}_{q^2}}/\mu_{q+1}$. We identify as $\mu_{q+1} \simeq \mu_{q+1}^2/H ;\ \zeta \mapsto 
(\zeta,1)$.  
We have $\mathscr{K}'_{\chi} \simeq \overline{\varphi}^\ast 
\left(\mathscr{K}_{Y_n,\chi} \boxtimes \mathscr{K}_{\mathbb{G}_{\mathrm{m}},\chi^{-1}}\right)$. 
By the K\"{u}nneth formula and \eqref{som}, 
we have isomorphisms
\begin{align*}
 H_{\rm c}^i(U_{\overline{\mathbb{F}}_q},
 \pi^\ast \mathscr{L}_{\psi})[\chi] & \simeq 
 H_{\rm c}^i(U_{\overline{\mathbb{F}}_q}/\mu_{q+1},
 \mathscr{K}'_{\chi} \otimes (\mathrm{pr}_2\circ \overline{\varphi})^\ast 
 \mathscr{L}_{\psi}) \\
 & \simeq H_{\rm c}^i(U_{\overline{\mathbb{F}}_q}/\mu_{q+1},
 \overline{\varphi}^\ast(\mathscr{K}_{Y_n,\chi} \boxtimes 
 (\mathscr{K}_{\mathbb{G}_{\mathrm{m}},\chi^{-1}} \otimes \mathscr{L}_{\psi}))) \\
 & \simeq 
 H_{\rm c}^i (Y_{n,\overline{\mathbb{F}}_q}
 \times \mathbb{G}_{\mathrm{m}},\mathscr{K}_{Y_n,\chi} \boxtimes 
 (\mathscr{K}_{\mathbb{G}_{\mathrm{m}},\chi^{-1}} \otimes \mathscr{L}_{\psi})) \\
 & \simeq H_{\rm c}^{i-1}(Y_{n,\overline{\mathbb{F}}_q},\mathscr{K}_{Y_n,\chi}) 
 \otimes 
 H_{\rm c}^1(\mathbb{G}_{\mathrm{m}},\mathscr{K}_{\mathbb{G}_{\mathrm{m}},\chi^{-1}} \otimes 
 \mathscr{L}_{\psi}) 
 \simeq H_{\rm c}^{i-1}(Y_{n,\overline{\mathbb{F}}_q},\mathscr{K}_{Y_n,\chi})  
\end{align*}
as $\oU_n (\bF_q)$-representations.  
The last claim follows from the above arguments and 
\cite[Sommes trig., Proposition 4.2 (ii)]{DelCoet}. 
\end{proof}

We study the cohomology of $Z_{\overline{\mathbb{F}}_q}$. 
Let $Z^0 \subset Z$ be the open subscheme defined by
$(x_1,\ldots,x_n) \neq \mathbf{0}=(0,\ldots,0)$. 
Now, we regard $Z^0$ as a closed subscheme of 
$\mathbb{A}_{\mathbb{F}_q}^n \setminus \{\mathbf{0}\}$. 
The morphism $\mathbb{A}_{\mathbb{F}_q}^n
\setminus \{\mathbf{0}\} \to \mathbb{P}^{n-1}_{\mathbb{F}_q};\ (x_i)_{1 \leq i \leq n} \mapsto 
[x_1:\cdots:x_n]$ is a $\mathbb{G}_{\mathrm{m},\mathbb{F}_q}$-bundle. 
By restricting this to $Z^0$, 
we have a $\mathbb{G}_{\mathrm{m},\mathbb{F}_q}$-bundle 
\[
\pi_0 \colon Z^0 \to S_n;\ (x_i)_{1 \leq i \leq n} \mapsto 
[x_1:\cdots:x_n].   
\]
The morphism $\pi_0$ factors as  
\[
 Z^0 \xrightarrow{f_{Z^0}} Z^0/\mu_{q+1} \xrightarrow{\overline{\pi}_0} S_n , 
\]
where $Z^0/\mu_{q+1}$ denotes 
the quotient of $Z^0$ under the natural action of 
the group scheme of $(q+1)$-st roots of unity over $\bF_q$. 

\begin{lem}\label{sppe}
\begin{enumerate}
\item
We have the spectral sequence 
\begin{equation}\label{spectral}
E_2^{a,b}=H^a(S_{n,\overline{\mathbb{F}}_q},R^b\overline{\pi}_{0,!} \Lambda) \Rightarrow
E^{a+b}=H_{\rm c}^{a+b}(Z^0_{\overline{\mathbb{F}}_q},\Lambda). 
\end{equation}
\item\label{enu:Rpi} 
We have 
\[
 R^b\overline{\pi}_{0,!}
 \Lambda \simeq 
 \begin{cases}
 \Lambda & \textrm{if $b=1$}, \\ 
 \Lambda (-1) & \textrm{if $b=2$}, \\ 
 0 & \textrm{if $b \neq 1,2$.} 
 \end{cases}
\]
\item\label{enu:HZ0} 
The action of $\mu_{q+1}$ on 
$H_{\rm c}^i(Z^0_{\overline{\mathbb{F}}_q},\Lambda)$ is trivial for any $i$. 
\end{enumerate}
\end{lem}
\begin{proof}
Since the morphism $\overline{\pi}_0$ is 
a $\mathbb{G}_{\mathrm{m}}$-bundle 
and any $\mathbb{G}_{\mathrm{m}}$-bundle 
is the complement of the zero section in a line bundle, 
we have the second assertion by 
\cite[VII, Proposition 1.3(ii)]{SGA5}. 

We show the first assertion. 
We have the Leray spectral sequence 
\[
E_2^{a,b}=H^a(S_{n,\overline{\mathbb{F}}_q},R^b\pi_{0,!} \Lambda ) \Rightarrow
E^{a+b}=H_{\rm c}^{a+b}(Z^0_{\overline{\mathbb{F}}_q},\Lambda ).   
\]
For $\chi \in \Hom (\mu_{q+1},\Lambda^{\times})$, 
let $\mathscr{K}^0_{\chi}$ denote the 
smooth $\Lambda$-sheaf on  
$Z^0_{\mathbb{F}_{q^2}}/\mu_{q+1}$ associated to the finite Galois \'{e}tale covering 
$f_{Z^0}$ and $\chi^{-1}$. 
We have $f_{Z^0,!} \Lambda \simeq 
\bigoplus_{\chi \in \mu_{q+1}^{\vee}} \mathscr{K}^0_{\chi}$.  
We have $R\overline{\pi}_{0,!} \mathscr{K}^0_{\chi}=0$
for $\chi \in \Hom (\mu_{q+1},\Lambda^{\times}) \setminus \{1\}$ by 
\cite[Sommes trig., Th\'{e}or\`{e}me 2.7]{DelCoet}. 
Hence we have $R^b\pi_{0,!} \Lambda \simeq R^b\overline{\pi}_{0,!}( f_{Z^0,!} \Lambda ) \simeq R^b\overline{\pi}_{0,!} \Lambda$ for any $b$.
Hence the first claim follows. 
The third assertion follows from the proof of 
the first one. 
\end{proof}
\begin{lem}\label{fer}
The action of $\mu_{q+1}$ on 
$H_{\rm c}^i(Z_{\overline{\mathbb{F}}_q},\Lambda )$ is trivial for any $i$.
\end{lem}
\begin{proof}
We have an exact sequence 
\begin{equation}\label{eq:exZ0Z}
 0 \to 
 H_{\rm c}^0(Z_{\overline{\mathbb{F}}_q},\Lambda ) 
 \to 
 H_{\rm c}^0(\{ \mathbf{0} \}_{\overline{\mathbb{F}}_q},\Lambda ) 
 \to 
 H_{\rm c}^1(Z^0_{\overline{\mathbb{F}}_q},\Lambda ) 
 \to
 H_{\rm c}^1(Z_{\overline{\mathbb{F}}_q},\Lambda ) 
 \to 0 
\end{equation}
and an isomorphism 
\begin{equation}\label{eq:Z0Z}
H_{\rm c}^i(Z^0_{\overline{\mathbb{F}}_q},\Lambda ) \xrightarrow{\sim} 
 H_{\rm c}^i(Z_{\overline{\mathbb{F}}_q},\Lambda ) 
\end{equation}
for $i \geq 2$. 
Hence 
the assertion follows from 
Lemma \ref{sppe} \ref{enu:HZ0}. 
\end{proof}
\begin{cor}\label{fer2}
Let $\chi \in \Hom (\mu_{q+1},\Lambda^{\times}) \setminus \{1\}$. 
We have an isomorphism 
\[
H_{\rm c}^i(U_{\overline{\mathbb{F}}_q},\pi^\ast \mathscr{L}_{\psi})[\chi] \simeq 
H_{\rm c}^i(\mathbb{A}^n,\pi^\ast \mathscr{L}_{\psi})[\chi]
\]
for any $i$. 
Furthermore, if $i \neq n$, this is isomorphic to zero.  
\end{cor}
\begin{proof}
The former claim follows from Lemma \ref{iso2} and Lemma \ref{fer}. The latter claim follows from \eqref{XLpsi}, 
Theorem \ref{old} and Proposition \ref{prop:Xmodl}. 
\end{proof}

\begin{lem}\label{unip}
Assume that $\Lambda =\ol{\bQ}_{\ell}$. 
\begin{enumerate}
\item  
We have a $\Gal (\overline{\mathbb{F}}_q/\mathbb{F}_q)$-equivariant isomorphism 
\[
H_{\rm c}^{n-1}(Z_{\overline{\mathbb{F}}_q},\overline{\mathbb{Q}}_{\ell}) 
 \xrightarrow{\sim}
H_{\rm c}^n(U_{\overline{\mathbb{F}}_q},\pi^\ast \mathscr{L}_{\psi})[1_{\mu_{q+1}}].
\]
\item\label{enu:HAY}
Assume that $n \geq 2$. 
We have a $\Gal (\overline{\mathbb{F}}_q/\mathbb{F}_q)$-equivariant 
isomorphism 
\[
 H_{\rm c}^n(\mathbb{A}^n,\pi^\ast \mathscr{L}_{\psi})[1_{\mu_{q+1}}]  \xrightarrow{\sim} 
 H_{\rm c}^{n-1}
 (Y_{n,\overline{\mathbb{F}}_q},\overline{\mathbb{Q}}_{\ell}(-1)). 
\]
\end{enumerate}
\end{lem}
\begin{proof} 
By Lemma \ref{fol}, we have an isomorphism 
\begin{equation}\label{iso1}
 H_{\rm c}^n(U_{\overline{\mathbb{F}}_q},\pi^\ast \mathscr{L}_{\psi})[1_{\mu_{q+1}}]
\simeq H_{\rm c}^{n-1}(Y_{n,{\overline{\mathbb{F}}_q}},\overline{\mathbb{Q}}_{\ell}). 
\end{equation}
Let $\phi_n$ be the unipotent representation of 
$\oU_n (\bF_q)$ corresponding to the partition 
$(n-1,1)$ of $n$ (\cf \S \ref{reviewu}). 
By \cite[the proof of Theorem 1]{HoMaTT}, 
we have 
\begin{equation}\label{iso3}
H_{\rm c}^{n-1}(Y_{n,\overline{\mathbb{F}}_q},\overline{\mathbb{Q}}_{\ell}) \simeq \phi_n.
\end{equation}
Note that differentials between $E_2$-terms of 
\eqref{spectral} are zero except 
\[
 d_2^{m-2,2} \colon E_2^{m-2,2}=H^{m-2}(S_{n,\overline{\mathbb{F}}_q},\overline{\mathbb{Q}}_{\ell}(-1))
\to 
E_2^{m,1}=H^{m}(S_{n,\overline{\mathbb{F}}_q},\overline{\mathbb{Q}}_{\ell}) 
\]
for $m \in \bZ$. 

We show the first assertion. 
We have 
$H_{\rm c}^{n-1}(Z_{\overline{\mathbb{F}}_q},\overline{\mathbb{Q}}_{\ell}) \neq 0$ by Lemma \ref{sppe}, \eqref{eq:exZ0Z}, 
\eqref{eq:Z0Z} and \cite[Theorem 1]{HoMaTT}. 
By \eqref{iso1} and \eqref{iso3}, 
the $\oU_n (\bF_q)$-representation 
$H_{\rm c}^n(U_{\overline{\mathbb{F}}_q},\pi^\ast \mathscr{L}_{\psi})[1_{\mu_{q+1}}]$ 
is irreducible. 
Hence the claim follows from 
Lemma \ref{iso2} and Lemma \ref{fer}. 

We show the second assertion. 
By Lemma \ref{fol}, we have isomorphisms 
\begin{equation}\label{iso1'}
H_{\rm c}^{n+1}(U_{\overline{\mathbb{F}}_q},\pi^\ast \mathscr{L}_{\psi})[1_{\mu_{q+1}}] \simeq 
 H_{\rm c}^{n}(Y_{n,{\overline{\mathbb{F}}_q}},\overline{\mathbb{Q}}_{\ell}) \simeq \begin{cases}
 0 & \textrm{if $n \geq 3$}, \\
 \overline{\mathbb{Q}}_{\ell}(-1) & \textrm{if $n=2$},  
 \end{cases}
\end{equation}
where we use \cite[the last line in p.257]{HoMaTT} if $n \geq 3$. 
By the first assertion
and Lemma \ref{iso2}, 
we have an isomorphism 
\begin{equation}\label{aa}
H_{\rm c}^n(\mathbb{A}^n,\pi^\ast \mathscr{L}_{\psi})[1_{\mu_{q+1}}] \xrightarrow{\sim} 
\Ker (\delta \colon H_{\rm c}^n(Z_{\overline{\mathbb{F}}_q},\overline{\mathbb{Q}}_{\ell}) \to 
H_{\rm c}^{n+1}(U_{\overline{\mathbb{F}}_q},\pi^\ast \mathscr{L}_{\psi})[1_{\mu_{q+1}}]). 
\end{equation}
Consider the composite 
\begin{gather}\label{eq:comp}
\begin{aligned}
 H_{\rm c}^n(\mathbb{A}^n,\pi^\ast \mathscr{L}_{\psi})[1_{\mu_{q+1}}] 
 &\to 
 H_{\rm c}^n(Z_{\overline{\mathbb{F}}_q},\overline{\mathbb{Q}}_{\ell}) 
 \simeq 
 H_{\rm c}^n(Z_{\overline{\mathbb{F}}_q}^0,\overline{\mathbb{Q}}_{\ell}) \\ 
 &\to 
 H^{n-2}(S_{n,\overline{\mathbb{F}}_q},\overline{\mathbb{Q}}_{\ell}(-1)) 
 \to 
 H_{\rm c}^{n-1}
 (Y_{n,\overline{\mathbb{F}}_q},\overline{\mathbb{Q}}_{\ell}(-1))
\end{aligned}
\end{gather}
where the third morphism is induced by 
the spectral sequence \eqref{spectral} and 
the last morphism is as in \cite[(i),(ii) in p.258]{HoMaTT}. 
Then the kernels of the morphisms in 
\eqref{eq:comp} are direct sums of 
trivial $\oU_n (\bF_q)$-representations 
by \eqref{aa} and 
\cite[Theorem 1 and (i),(ii) in p.258]{HoMaTT}. 
On the other hand, 
the source and the target of \eqref{eq:comp} 
are irreducible and non-trivial by Lemma \ref{special} and \eqref{iso3}. 
Therefore \eqref{eq:comp} is an isomorphism. 
\end{proof}

\begin{prop}\label{cohXY}
Assume that $\Lambda =\ol{\bQ}_{\ell}$. 
For $\chi \in \mu_{q+1}^{\vee}$, we have 
\[ 
 (H_{\rm c}^n(X_{n,\overline{\mathbb{F}}_q},\overline{\mathbb{Q}}_{\ell})[\psi])[\chi] 
 =
 \begin{cases}
 (-1)^{n-1} \sum_{i \geq 0} (-1)^i H_{\rm c}^i(Y_{n,\overline{\mathbb{F}}_q},\mathscr{K}_{Y_n,\chi}) & \textrm{if $\chi \neq 1$}, \\ 
 (-1)^n( 1_{\oU_n (\bF_q)}- \sum_{i \geq 0} (-1)^i H_{\rm c}^i(Y_{n,\overline{\mathbb{F}}_q},\overline{\mathbb{Q}}_{\ell})) & \textrm{if $\chi = 1$}
 \end{cases} 
\] 
as representations of $\oU_n (\bF_q)$. 
\end{prop}
\begin{proof}
If $\chi \neq 1$, the claim follows from 
\eqref{XLpsi}, Lemma \ref{fol} and 
Corollary \ref{fer2}. 
If $\chi = 1$ and $n=1$, the claim is trivial. 
If $\chi = 1$ and $n \geq 2$, 
the claim follows from 
\eqref{XLpsi}, Lemma \ref{unip} \ref{enu:HAY} and the fact that 
$H_{\rm c}^i(Y_{n,\overline{\mathbb{F}}_q},\overline{\mathbb{Q}}_{\ell})=0$ 
if $i \neq n-1, 2(n-1)$ (\cf \cite[Proof of Theorem 1]{HoMaTT}). 
\end{proof}

\section{Lusztig induction}\label{sec:Lus}
We recall Lusztig induction in \cite{LusFinunip}. 
Let $G$ be a connected reductive group defined over $\overline{\bF}_q$ 
with an $\bF_q$-rational structure. 
Let $F$ be the corresponding Frobenius endomorphism on $G$. 

Let $P$ be a parabolic subgroup of $G$, 
and $M$ be a Levi subgroup of $P$. 
Assume that $M$ is $F$-stable. 
We put 
\[
 \widetilde{Y}_P=\{g \in G \mid 
 g^{-1} F(g) \in F(U_P) \}. 
\]
Let $G^F \times M^F$ act on $\widetilde{Y}_P$ by 
$\widetilde{Y}_P \to \widetilde{Y}_P;\ 
x \mapsto g x g'^{-1}$ for $(g,g') \in G^F \times M^F$. 
We set 
\begin{align*}
 Y_P = \left\{g \in  G \mid g^{-1} F(g) \in F(P) \right\}/M.
\end{align*}
The morphism 
\[
 \pi_P \colon \widetilde{Y}_P \to Y_P ;\ g \mapsto gM 
\]
is known to be an 
$M^F$-torsor (\cf \cite[\S 7.3]{DiMiParDL}).
We give a proof of it.  
\begin{lem}\label{lem:Mtor}
The morphism $\pi_P$ is an $M^F$-torsor. 
\end{lem}
\begin{proof}
Let $X=\left\{g \in G \mid g^{-1} F(g) \in F(P)\right\}$. 
We regard $\widetilde{Y}_P$ as a closed subscheme of $X$. 
We consider the morphisms 
$\phi' \colon F(U_P) \times M \to F(P);\ (u,m) \mapsto 
m^{-1} u F(m)$ and 
$\phi \colon 
\widetilde{Y} \times M \to X;\ (g,m) \mapsto gm$. 
Let $L \colon G \to G;\ g \mapsto g^{-1} F(g)$. 
We have the cartesian diagram
\[
\xymatrix{
\widetilde{Y}_P \times M \ar[r]^{\phi\!\!\!\!\!\!}\ar[d]_{L \times 1} & 
X \ar[d]^{L} \\
F(U_P) \times M \ar[r]^{\phi'\!\!\!\!\!\!\!\!\!\!} & F(P). 
}
\]
We consider the morphisms
\begin{align*}
 \phi_1 & \colon F(U_P) \times M \to F(U_P) \times M;\ 
 (u,m) \mapsto (m^{-1}um,m ), \\
 \phi_2 & \colon F(U_P) \times M \to F(U_P) \times M;\ 
 (u,m) \mapsto (u,m^{-1}F(m)), \\
 \phi_3 & \colon F(U_P) \times M \to F(P);\ (u,m) \mapsto um. 
\end{align*}
We can easily check $\phi'=\phi_3 \circ \phi_2 \circ \phi_1$. 
The morphisms $\phi_i$ for $i=1,3$ are isomorphisms. 
The morphism 
$\phi_2$ is an $M^F$-torsor. 
Hence, $\phi$ is an $M^F$-torsor. 
By taking the quotient of $\phi$ by $M$, we obtain the claim. 
\end{proof}

For a representation $\rho$ of $M^F$, 
we consider a virtual $G^F$-module 
\[
 R^{G}_{M \subset P}(\rho)
 =\sum_i(-1)^i ( 
 H_{\rm c}^i(\widetilde{Y}_P,\overline{\mathbb{Q}}_{\ell}) 
 \otimes \rho )^{M^F} 
\]
which is introduced in \cite[p.203]{LusFinunip}. 

Assume that $M=P \cap F(P)$. 
Then we have 
\[
 Y_P =\left\{g \in G \mid g^{-1} F(g) \in F(P) \right\}/(P \cap F(P))
 =\left\{g P \in G/P \mid g^{-1} F(g) \in PF(P) \right\}. 
\]
Assume that there is a decomposition $M=T \times M'$ as algebraic groups over $\bF_q$ 
where $T$ is a torus and $M'$ is a reductive group. 
Let $\chi \in (T^{F})^{\vee}$. 
Let $\mathscr{K}_{Y_P,\chi}$ be the smooth 
$\overline{\mathbb{Q}}_{\ell}$-sheaf on $Y_P$ 
associated to $\chi^{-1}$ under the natural projection 
$M^{F} \to T^{F}$ and $\pi_P$. 
We have 
\begin{gather}\label{ttp}
\begin{aligned}
 R^{G}_{M \subset P}(\chi \otimes 1_{M'^F})
 &=\sum_{i}(-1)^i H_{\rm c}^i(\widetilde{Y}_P,\overline{\mathbb{Q}}_{\ell}) 
 [\chi^{-1} \otimes 1_{M'^F}]\\
 &=
 \sum_{i}(-1)^i H_{\rm c}^i\bigl(
 \widetilde{Y}_P/M'^{F},\overline{\mathbb{Q}}_{\ell}\bigr) 
 [\chi^{-1}] 
 =\sum_{i}(-1)^i H_{\rm c}^i(
 Y_P,\mathscr{K}_{Y_P,\chi}).  
\end{aligned}
\end{gather}
By \cite[Corollary 7.14]{DeLuRep}, we have 
\[
 1_{M'^F}=\sum_{(T') \subset M'} \frac{1}{|W(T')^F|} R_{T'}^{M'}(1_{T'^F})
\]
as characters of $M'^F$, where 
the summation runs over all $M'^F$-conjugacy 
classes of $F$-stable maximal tori $T'$ of $M'$. 
We have 
\[
 \chi \otimes 1_{M'^F}=\sum_{(T') \subset M'} 
 \frac{1}{|W(T')^F|} R_{T \times T'}^{M}(\chi \otimes 1_{T'^F})
\]
as characters of $M^F$. 
By applying $R^G_{M \subset P}$ to this, we have 
\begin{equation}\label{vv}
 R^{G}_{M \subset P}(\chi \otimes 1_{M'^F})=\sum_{(T') \subset M'} \frac{1}{|W(T')^F|}
 R^{G}_{T \times T'}(\chi \otimes 1_{T'^F})
\end{equation}
as characters of $G^F$ by \cite[Corollary 5]{LusFinunip}. 
Therefore we have 
\begin{equation}\label{eq:HYR}
 \sum_{i}(-1)^i H_{\rm c}^i(
 Y_P,\mathscr{K}_{Y_P,\chi}) 
 = \sum_{(T') \subset M'} \frac{1}{|W(T')^F|}
 R^{G}_{T \times T'}(\chi \otimes 1_{T'^F}). 
\end{equation}

\section{Unitary group}\label{sec:uni}
Let $\psi \in \mathbb{F}_q^{\vee} \setminus \{1\}$ in the rest of this paper. 

\subsection{Unipotency}
We use notation in Subsection \ref{reviewu}. 
Let $P \subset G_n$ be the parabolic subgroup 
consisting of matrices $(x_{i,j})$ such that 
$x_{i,1}=0$ for $2 \leq i \leq n$.  
Let $M$ be the $F$-stable 
Levi subgroup of $P$. Note that $P$ is not 
stable by $F$. 
We have $M \simeq G_1 \times 
G_{n-1}$ diagonally embedded in $G_n$.  
We have $G_n^F=\oU_n (\bF_q)$ and 
$M^F=\oU_1 (\bF_q) \times \oU_{n-1}(\bF_q)$. 

For $x=(x_{i,j}) \in \widetilde{Y}_P$, 
we have $\sum_{i=1}^n x_{i,1}^{q+1}=1$ by 
$F^{-1}(x)x \in F(U_P)$. 
Hence, we have the morphism 
$\widetilde{\pi}_n \colon \widetilde{Y}_P \to 
 \widetilde{Y}_{n,\overline{\mathbb{F}}_{q}};\ 
 x \mapsto (x_{i,1})_{1 \leq i \leq n}$. 
This induces a morphism 
$\widetilde{\varphi}_n \colon 
 \widetilde{Y}_P/\oU_{n-1}(\bF_q) \to 
 \widetilde{Y}_{n,\overline{\mathbb{F}}_{q}}$. 
We have an identification 
$\widetilde{Y}_P/M^F \simeq Y_P$ by Lemma \ref{lem:Mtor}. 
As shown in \cite[p.259]{HoMaTT}, 
$\widetilde{\pi}_n$ induces an isomorphism 
$\varphi_n \colon Y_P \simeq \widetilde{Y}_P/M^F 
 \xrightarrow{\sim} Y_{n,\overline{\mathbb{F}}_{q}}$. 
We have the commutative diagram 
\[
 \xymatrix{
 \widetilde{Y}_P \ar[d] \ar[rd]^-{\widetilde{\pi}_n} & \\ 
 \widetilde{Y}_P/\oU_{n-1}(\bF_q) \ar[r]_-{\widetilde{\varphi}_n} 
 \ar[d]_{f_{Y_P}} 
 & \widetilde{Y}_{n,\overline{\mathbb{F}}_{q}} \ar[d]^{f_{Y_n}} \\ 
 Y_P \ar[r]_-{\varphi_n} & Y_{n,\overline{\mathbb{F}}_{q}} 
}
\]
where the vertical morphisms are natural ones. 
Since 
$f_{Y_P}$ and $f_{Y_n}$ are $\mu_{q+1}$-torsors 
and $\varphi_n$ is an isomorphism, 
$\widetilde{\varphi}_n$ is also an isomorphism. 

\begin{prop}\label{class2}
We have 
\[
(H_{\rm c}^n(X_{n,\overline{\mathbb{F}}_q},\overline{\mathbb{Q}}_{\ell})
 [\psi])[\chi]
 =
\begin{cases}
(-1)^{n-1}
R^{G_n}_{M \subset P}(\chi \otimes 1_{G_{n-1}^F}) 
& \textrm{if $\chi \in \mu_{q+1}^{\vee} \setminus \{1\}$}, 
\\
(-1)^{n}(1_{G_n^F}-R^{G_n}_{M \subset P}(1_{M^F})) & \textrm{if $
\chi=1$} 
\end{cases}
\]
as characters of $G_n^F$. 
\end{prop}
\begin{proof}
Since $\widetilde{\varphi}_n$ is an isomorphism, 
the claim follows from Proposition \ref{cohXY} and \eqref{ttp}. 
\end{proof}

In the sequel, 
we write $\omega_{\oU_n}$ for $\omega_{\oU (\bF_{q^2}^n,h_n)}$ defined in \eqref{eq:omegadef}. 
We give a geometric proof of the following fact 
using a relation between $X_n$ and 
a certain Deligne--Lusztig variety. 

\begin{cor}\label{cc}
The $\oU_n (\bF_q)$-representation 
$\omega_{\oU_n} [1_{\mu_{q+1}}]$ is unipotent. 
For each $\chi \in \mu_{q+1}^{\vee} \setminus \{1\}$, 
the $\oU_n (\bF_q)$-representation 
$\omega_{\oU_n} [\chi]$ is not unipotent. 
\end{cor}
\begin{proof}
The claim follows from \eqref{vv} and 
Proposition \ref{class2} 
(\cf \cite[the sentence after Definition 7.8]{DeLuRep}). 
\end{proof}

\begin{rem}
This is a special case of the preservation of unipotency under the Howe correspondence for a unitary pair, 
which is proved in \cite{AdMoUnip} if $q$ is odd and large enough. 
The assumption on $q$ is necessary in \cite{AdMoUnip}, 
since the proof depends on \cite{SriWcl}. 
\end{rem}

\begin{rem}\label{remuni}
Let $\epsilon_G$ be the $\mathbb{F}_q$-rank of 
a linear algebraic group $G$ over $\overline{\mathbb{F}}_q$. 
For any integer $m \geq 1$, let $m_p$ be the largest 
power of $p$ dividing $m$ and $m_{p'}=m/m_p$.
We go back to our situation. 
We have $\epsilon_{G_n}=[n/2]$ (\cf 
\cite[\S15.1]{DiMiRepLie}). 
Hence, we have $\epsilon_{G_n}\epsilon_M=(-1)^{n-1}$. By \cite[Proposition 12.17]{DiMiRepLie}, we have 
\[
\dim R_{M \subset P}^{G_n}(\chi \otimes 1_{G_{n-1}^F})
=\epsilon_{G_n} \epsilon_{M} \frac{|G_n^F|_{p'}}{|M^F|_{p'}}
=(-1)^{n-1} \frac{q^n-(-1)^{n}}{q+1}
\]
for any $\chi \in \mu_{q+1}^{\vee}$. 
This is compatible with Lemma \ref{special} and Proposition \ref{class2}. 
\end{rem}

\subsection{Branching formula}

We consider the embedding 
$\iota \colon \oU_n \hookrightarrow \oU_{n+1};\ g \mapsto \mathrm{diag}(g,1)$. 
Let $\chi \in \mu_{q+1}^{\vee}$. 
We regard the 
$\oU_{n+1} (\bF_q)$-representation 
$\omega_{\oU_{n+1}} [\chi]$ as a $\oU_n (\bF_q)$-representation
by $\iota$. 
We give a geometric proof of 
the following branching formula 
(\cf \cite[Lemma 4.4]{TieWglirr} in $\SU_n$ case). 
\begin{prop}\label{prop:bran}
Let $\chi \in \mu_{q+1}^{\vee}$. 
We have a decomposition 
\[
 \omega_{\oU_{n+1}} [\chi] |_{\oU_n (\bF_q)} = 
 \bigoplus_{\chi' \in \mu_{q+1}^{\vee} \setminus \{\chi\}} 
 \omega_{\oU_n} [\chi']
\]
as $\oU_n (\bF_q)$-representations. 
\end{prop}
\begin{proof}
By \eqref{XLpsi}, we have isomorphisms 
\begin{align*}
H_{\mathrm{c}}^{n+1}(X_{n+1,\overline{\mathbb{F}}_q},\overline{\mathbb{Q}}_{\ell})[\psi]
& \simeq 
 H_{\rm c}^{n+1}(\mathbb{A}^{n+1},\mathscr{L}_{\psi}(x_1^{q+1}+\cdots+x_{n+1}^{q+1})) \\
& \simeq 
 H_{\rm c}^{n}(\mathbb{A}^n,  \mathscr{L}_{\psi}(x_1^{q+1}+\cdots+x_n^{q+1})) \otimes 
  H_{\rm c}^1(\mathbb{A}^1, \mathscr{L}_{\psi}(x_{n+1}^{q+1})) \\
& \simeq H_{\mathrm{c}}^{n}(X_{n,\overline{\mathbb{F}}_q},\overline{\mathbb{Q}}_{\ell})[\psi] \otimes 
H_{\mathrm{c}}^1(C_{\overline{\mathbb{F}}_q},\overline{\mathbb{Q}}_{\ell})[\psi], 
\end{align*}
where we use the K\"{u}nneth formula at the second isomorphism. 
By taking the $\chi$-isotypic part of the above isomorphism, we obtain 
the assertion by Lemma \ref{tri} \ref{enu:H1C}  and Theorem \ref{old}. 
\end{proof}

\section{Symplectic orthogonal pair}\label{sec:sop}
\subsection{Representation of a dual pair}\label{ob}
\subsubsection{Construction}
We use the same notation as \S \ref{another}. 
Let $\omega$ be a symplectic form on $V$ 
obtained by the restriction of $h_{2n}'$ to $V$. 
By regarding $V_{\mathbb{F}_{q^2}}$
as a vector space over $\mathbb{F}_q$, 
we consider the symplectic form 
associated to $h_{2n}'$: 
\[
 \omega' \colon V_{\mathbb{F}_{q^2}} \times 
V_{\mathbb{F}_{q^2}} \to \mathbb{F}_q;\ 
(v,v') \mapsto 
\Tr_{\mathbb{F}_{q^2}/\mathbb{F}_q}(h_{2n}'(v,v')). 
\]
Let $\Sp(V,\omega)$ and 
$\Sp(V_{\mathbb{F}_{q^2}},\omega')$
be the isometry groups of $(V,\omega)$ and $(V_{\mathbb{F}_{q^2}},\omega')$, respectively. 
We have 
$\Sp(V,\omega) \subset \oU (V_{\mathbb{F}_{q^2}},h_{2n}')
\subset \Sp(V_{\mathbb{F}_{q^2}},\omega')$. 

We put $W=\mathbb{F}_{q^2}$. 
Let $h_1 \colon W \times W \to \mathbb{F}_{q^2};\ 
(x,y) \mapsto x^q y$. 
We have $\oU (W,h_1)=\mu_{q+1}$. 
Under the natural isomorphism 
$V \otimes_{\bF_q} W \simeq V_{\bF_{q^2}}$, the skew-hermitian form 
$\omega \otimes h_1$ corresponds to $h_{2n}'$. 
This induces a morphism 
\begin{equation}\label{sub}
\Sp(V,\omega) \times \oU (W_,h_1) \hookrightarrow 
\oU (V_{\mathbb{F}_{q^2}},h_{2n}'). 
\end{equation}

We have natural actions of 
$\Gal (\bF_{q^2}/\bF_q) \simeq \bZ/2\bZ$ on 
$\oU (W,h_1)$ and $\oU (V_{\mathbb{F}_{q^2}},h_{2n}')$. 
Then the homomorphism \eqref{sub} extends naturally 
to 
\[
 \Sp(V_,\omega) \times (\oU (W,h_1) \rtimes 
 (\mathbb{Z}/2\mathbb{Z})) \hookrightarrow 
 \oU (V_{\mathbb{F}_{q^2}},h_{2n}') \rtimes (\mathbb{Z}/2\mathbb{Z}). 
\] 
The $\mathbb{F}_q$-automorphism 
$F \colon V_{\mathbb{F}_{q^2}} \to V_{\mathbb{F}_{q^2}};\ 
 (v_i) \mapsto (v_i^q)$ 
preserves $\omega'$. We regard this as an element of 
$\Sp(V_{\mathbb{F}_{q^2}},\omega')$. 
Then we have the injective homomorphism 
\begin{equation}\label{phi}
 \oU (V_{\mathbb{F}_{q^2}},h_{2n}') \rtimes (\mathbb{Z}/2\mathbb{Z}) \hookrightarrow \Sp(V_{\mathbb{F}_{q^2}},\omega');\ (g,k) \mapsto gF^k. 
\end{equation}
Let $\mathbb{Z}/2\mathbb{Z}$ act on 
$H_{\rm c}^{2n}(X_{2n,\overline{\mathbb{F}}_q}',\overline{\mathbb{Q}}_{\ell})(n)[\psi]$
by 
\[
\mathrm{Fr}_q^k \colon H_{\rm c}^{2n}(X_{2n,\overline{\mathbb{F}}_q}',\overline{\mathbb{Q}}_{\ell})(n)[\psi] \to 
H_{\rm c}^{2n}(X_{2n,\overline{\mathbb{F}}_q}',\overline{\mathbb{Q}}_{\ell})(n)[\psi]
\]
for $k \in \mathbb{Z}/2\mathbb{Z}$, 
which is well-defined by Lemma \ref{lem:F2triv}. 
By the natural action of 
$\Gal (\bF_{q^2}/\bF_q)$ on $V_{\bF_{q^2}}$ and $\bF_{q^2}$, 
we have a natural action of 
$\Gal (\bF_{q^2}/\bF_q) \simeq \bZ/2\bZ$ on 
$\HU(h_{2n}')$. 
We write $\pi_{h_{2n}',\psi}$ for the 
$\HU(h_{2n}') \rtimes (\mathbb{Z}/2\mathbb{Z})$-representation 
$H_{\rm c}^{2n}(X_{2n,\overline{\mathbb{F}}_q}',\overline{\mathbb{Q}}_{\ell})(n)[\psi]$. 
Then the action of $1 \in \mathbb{Z}/2\mathbb{Z}$ 
induces an isomorphism between 
$\pi_{h_{2n}',\psi}[\chi]$ 
and 
$\pi_{h_{2n}',\psi}[\chi^{-1}]$ 
for $\chi \in \mu_{q+1}^{\vee} \setminus \{1\}$.

We consider the quadratic form 
$Q \colon W \to \mathbb{F}_q;\ x \mapsto x^{q+1}$. 
We set 
\[
 \oO (W,Q)=\{f \in \mathrm{Aut}_{\mathbb{F}_q}(W) \mid Q(f(w))=Q(w)\ \textrm{for $w \in W$}\}. 
\]
Since $Q(x)=h_1(x,x)$, 
we have a natural inclusion 
$\oU (W,h_1) \hookrightarrow \oO (W,Q)$. 
We regard 
$F_W \colon W \to W;\ x \mapsto x^q$ 
as an element of $\oO (W,Q)$. 
Then we have the isomorphism 
\[
 \oU (W,h_1) \rtimes (\mathbb{Z}/2\mathbb{Z})\xrightarrow{\sim} \oO (W,Q);\ (\xi,i) \mapsto \xi F_W^i, 
\]
since 
$\oO (W,Q)$
is isomorphic to the dihedral group $D_{2(q+1)}$ 
by \cite[Proposition 2.9.1]{KlLiSubcl}. 
We identify 
$\oU (W,h_1) \rtimes (\mathbb{Z}/2\mathbb{Z})$ 
with $\oO (W,Q)$ by the above isomorphism.

We consider the symmetric form $s_1 \colon W \times W \to \mathbb{F}_q;\ (x,y) \mapsto \Tr_{\mathbb{F}_{q^2}/\mathbb{F}_q}(h_1(x,y))$. Note that  
\[
Q(x+y)-Q(x)-Q(y)=s_1(x,y) \quad \textrm{for $x,y \in W$}. 
\]
Hence, we have the natural map $\oO (W,Q) 
\hookrightarrow \oO (W,s_1)$, which is an isomorphism if $q$ is odd. 
We have a natural homomorphism 
$\Sp(V,\omega)\times \oO (W,s_1)
\to \Sp(V_{\mathbb{F}_{q^2}},\omega')$, 
since 
$\omega \otimes s_1$ corresponds to $\omega'$ 
under the natural isomorphism 
$V \otimes_{\bF_q} W \simeq V_{\bF_{q^2}}$. 
We consider the composite  
\[
 \Sp(V,\omega)\times \oO (W,Q)
 \hookrightarrow 
 \Sp(V,\omega)\times \oO (W,s_1)
 \hookrightarrow \Sp(V_{\mathbb{F}_{q^2}},\omega'). 
\]
By the construction, 
we obtain the commutative diagram 
\begin{gather}\label{com}
\begin{aligned}
\xymatrix{
 \Sp(V,\omega)\times \oO (W,Q) \ar[r] & \Sp(V_{\mathbb{F}_{q^2}},\omega') \\
 \Sp(V,\omega) \times (\oU (W,h_1) \rtimes (\mathbb{Z}/2\mathbb{Z})) \ar[r] \ar@{=}[u] & \oU (V_{\mathbb{F}_{q^2}},h_{2n}') \rtimes (\mathbb{Z}/2\mathbb{Z}) \ar[u]
}
\end{aligned}
\end{gather}
where the going up arrows are natural homomorphisms. 
We put 
\[
 \omega_{\mathrm{SpO},n,\psi}=
 \pi_{h_{2n}',\psi}|_{\Sp(V,\omega)\times \oO (W,Q)} . 
\]

For a character $\chi_0 \in \mu_{q+1}^{\vee}$ such
that $\chi_0^2=1$, the $\chi_0$-isotypic part 
$\omega_{\mathrm{SpO},n,\psi}[\chi_0]$ is stable under the action of $\mathbb{Z}/2\mathbb{Z}$. 
For $\kappa \in \{ \pm 1 \}$, 
let $\omega_{\mathrm{SpO},n,\psi}[\chi_0]^{\kappa}$ denote the $\kappa$-eigenspace of $1 \in \mathbb{Z}/2\mathbb{Z}$
on $\omega_{\mathrm{SpO},n,\psi}[\chi_0]$. 
Let $\nu$ be the quadratic character of $\mu_{q+1}$ if $p \neq 2$. 

\begin{lem}\label{dim}
\begin{enumerate}
\item 
Assume that $p=2$. We have 
\begin{align*}
\dim \omega_{\mathrm{SpO},n,\psi}[1_{\mu_{q+1}}]^{\kappa}=\frac{(q^n-1)(q^n-q)}{2(q+1)}+\frac{1+\kappa}{2}q^n, \quad 
\dim \omega_{\mathrm{SpO},n,\psi}[\chi]=\frac{q^{2n}-1}{q+1}
\end{align*}  
for $\kappa \in \{\pm 1 \}$ and $\chi \in \mu_{q+1}^{\vee} \setminus \{1\}$. 
\item\label{enu:dimodd}
Assume that $p \neq 2$. 
We have 
\begin{align*}
\dim \omega_{\mathrm{SpO},n,\psi}[1_{\mu_{q+1}}]^{\kappa}&=\frac{(q^n-1)(q^n-q)}{2(q+1)}+\frac{1+\kappa}{2}q^n, \\  
\dim \omega_{\mathrm{SpO},n,\psi}[\nu]^{\kappa}&=
\frac{q^{2n}-1}{2(q+1)}, \quad 
\dim \omega_{\mathrm{SpO},n,\psi}[\chi]=\frac{q^{2n}-1}{q+1} 
\end{align*}  
for $\kappa \in \{\pm 1 \}$ and 
$\chi \in \mu_{q+1}^{\vee} \setminus \{1,\nu\}$.
\end{enumerate}
\end{lem}
\begin{proof}
We set 
$a_{n,\chi}=\dim \omega_{\mathrm{SpO},n,\psi}[\chi]$ 
for $\chi \in \mu_{q+1}^{\vee}$. 
We have a $\mu_{q+1}$-equivariant 
isomorphism $X_{2n,\mathbb{F}_{q^2}} \simeq X_{2n,\mathbb{F}_{q^2}}'$. 
Hence, we have 
\begin{equation}\label{n}
a_{n,\chi}=
\begin{cases}
\displaystyle \frac{q^{2n}+q}{q+1} & \textrm{if $\chi=1$}, \\[0.3cm]
\displaystyle  \frac{q^{2n}-1}{q+1} & \textrm{if $\chi \in \mu_{q+1}^{\vee} \setminus \{1\}$}
\end{cases}
\end{equation}
by Lemma \ref{special}. 

Let 
\[
 A=\biggl\{ (\chi_k)_{1 \leq k \leq n} \in (\mu_{q+1}^{\vee})^n \biggm| \prod_{k=1}^n \chi_k=1 \biggr\}. 
\]
Let $\mathbf{1} \in A$ denote the element 
$(\chi_k)_{1 \leq k \leq n}$ with $\chi_k=1$ for any 
$1 \leq k \leq n$. 
For each $(\chi_k)_{1 \leq k \leq n} \in A$, 
let $V_{(\chi_k)_{1 \leq k \leq n}}$ denote the subspace 
$\bigotimes_{k=1}^n \omega_{\mathrm{SpO},1,\psi}[\chi_k] 
\subset \omega_{\mathrm{SpO},n,\psi}[1_{\mu_{q+1}}]$. 
By Lemma \ref{ke}, 
we have 
$\dim \omega_{\mathrm{SpO},1,\psi}[1_{\mu_{q+1}}]=q$ 
and the action of 
$\mathbb{Z}/2\mathbb{Z}$ on $\omega_{\mathrm{SpO},1,\psi}[1_{\mu_{q+1}}]$ is trivial. 
Further, if $p \neq 2$, the dimensions of 
the $1$-eigenspace and the $(-1)$-eigenspace in 
$\omega_{\mathrm{SpO},1,\psi}[\nu]$ are same by Lemma \ref{ke}. 
Therefore, we have $\dim V_{\mathbf{1}}=q^n$ 
and the action of $\mathbb{Z}/2\mathbb{Z}$ on $V_{\mathbf{1}}$ 
is trivial. Further, 
the dimensions of 
the $1$-eigenspace and the $(-1)$-eigenspace in 
$\bigoplus_{(\chi_k)_{1 \leq i \leq n} \in A \setminus \{ \mathbf{1} \}}
V_{(\chi_k)_{1 \leq k \leq n}}$ 
are same. 
Hence, the dimension of 
$\omega_{\mathrm{SpO},n,\psi}[1_{\mu_{q+1}}]^{\kappa}$
equals 
\[
 \frac{a_{n,1} +\kappa \dim V_{\mathbf{1}}}{2}=
 \frac{1}{2}\left(\frac{q^{2n}+q}{q+1}+\kappa q^n\right)
 =\frac{(q^n-1)(q^n-q)}{2(q+1)} +\frac{1+\kappa}{2} q^n . 
\]
The second equality in the claim \ref{enu:dimodd} is proved similarly. 
\end{proof}

\begin{rem}
It is possible to calculate 
values of characters of representations in 
Lemma \ref{dim} using the geometric constructions in this paper and 
the Grothendieck--Lefschetz trace formula. 
See \cite[Proposition 4.9]{ITModlW} for an example of such a calculation. 
\end{rem}

\subsubsection{Compatibility}\label{sssec:comp}
In \S \ref{sssec:comp}, we always assume $p \neq 2$. 
Then we can construct a representation of 
the dual pair 
$\Sp(V,\omega)\times \oO (W,Q)$ 
as the restriction of the Weil representation of 
$\Sp(V_{\mathbb{F}_{q^2}},\omega')$. 
We show that the two constructions are compatible. 
We put 
\[
 \oH(V_{\bF_{q^2}},\omega')=V_{\bF_{q^2}} \times \bF_q 
\]
with multiplication 
\[
 (v,a)(v',a')=\left( v+v',a+a'+\frac{1}{2} \omega' (v,v') \right). 
\]
Then, $\Sp(V_{\mathbb{F}_{q^2}},\omega')$ 
acts on $\oH(V_{\bF_{q^2}},\omega')$ by 
$(v,a) \mapsto (gv,a)$ for 
$g \in \Sp(V_{\mathbb{F}_{q^2}},\omega')$ and 
$(v,a) \in \oH(V_{\bF_{q^2}},\omega')$. 
We put 
\[
 \HSp(\omega')=\oH(V_{\bF_{q^2}},\omega') \rtimes 
 \Sp(V_{\mathbb{F}_{q^2}},\omega'). 
\]
Let $\rho_{\HSp(\omega'),\psi}$ 
be the Heisenberg--Weil representation 
of $\HSp(\omega')$ associated to $\psi$ (\cf \cite[Theorem 2.4(a')]{GerWeil} and \cite[2.2]{GuHaGeoW}). 
We write 
$\rho_{\Sp(\omega'),\psi}$
for the restriction of $\rho_{\HSp(\omega'),\psi}$ to 
$\Sp(V_{\mathbb{F}_{q^2}},\omega')$. 
The representation $\rho_{\Sp(\omega'),\psi}$ 
is called the Weil representation of 
$\Sp(V_{\mathbb{F}_{q^2}},\omega')$ associated to 
$\psi$. 
We put 
\[
 \HSpU(h_{2n}')= 
 \oH(V_{\mathbb{F}_{q^2}},h_{2n}') \rtimes 
 (\Sp(V,\omega) \times \oU (W,h_1)). 
\]
By \eqref{phi} and the isomorphism 
\[
 \oH(V_{\mathbb{F}_{q^2}},h_{2n}') \xrightarrow{\sim} 
 \oH(V_{\bF_{q^2}},\omega');\ 
 (v,a) \mapsto \left( v,a -\frac{1}{2} h_{2n}'(v,v) \right) , 
\]
we have a morphism 
$\HU(h_{2n}') \rtimes (\bZ/2\bZ) \to \HSp(\omega')$. 
We have the commutative diagram 
\[
\xymatrix{
 \HSpU(h_{2n}') \rtimes (\bZ/2\bZ) \ar[r] 
 & \HU(h_{2n}') \rtimes (\bZ/2\bZ) \ar[r] & \HSp(\omega') \\
 \HSpU(h_{2n}') \ar[r] \ar[u] 
 & \HU(h_{2n}') \ar[u] 
 &  
}
\]
where the arrows in the square are natural homomorphisms. 
Let $F \in \Sp(V_{\mathbb{F}_{q^2}},\omega')$ be the $q$-th power Frobenius. 
Then $F$ takes $\rho_{\Sp(\omega'),\psi}[\chi]$ isomorphically to 
$\rho_{\Sp(\omega'),\psi}[\chi^{-1}]$
for $\chi \in \mu_{q+1}^{\vee}$. 
Let $\chi \in \mu_{q+1}^{\vee}$ such that 
$\chi^2=1$. 
For $\kappa \in \{ \pm 1 \}$, 
let $\rho_{\Sp(\omega'),\psi}[\chi]^{\kappa}$
denote the $\kappa$-eigenspace of $F$ on 
$\rho_{\Sp(\omega'),\psi}[\chi]$. 
\begin{lem}\label{spd}
Let $\eta_0$ be the quadratic character of $\mathbb{F}_q^{\times}$. 
Then we have $\Tr \rho_{\Sp(\omega'),\psi} (F)=(\eta_0(-1) q)^n$. 
\end{lem}
\begin{proof}
We imitate arguments in \cite[Lemma 2.1]{DesExcSp}. 
Let $L=\mathbb{F}_q^n \oplus \{\mathbf{0}\}$
and $L'=\{\mathbf{0}\} \oplus \mathbb{F}_q^n$ in $V$. 
These are maximal totally isotropic subspaces with respect to $\omega$. Then $\rho_{\Sp(\omega'),\psi}$ is realized in the space  $\overline{\mathbb{Q}}_{\ell}[L'_{\mathbb{F}_{q^2}}]$ of $\overline{\mathbb{Q}}_{\ell}$-valued 
functions on $L'_{\mathbb{F}_{q^2}}$ as in 
\cite[(3) in the proof of Theorem 2.4]{GerWeil}. 
Let $\sigma \colon \mathbb{F}_{q^2} \to \mathbb{F}_{q^2};\ x \mapsto x^q$. 
Since $F$ stabilizes $L_{\mathbb{F}_{q^2}}$ and $L'_{\mathbb{F}_{q^2}}$, we have 
$(\rho_{\Sp(\omega'),\psi}(F)f)(v)=\eta_0(-1)^n f((1 \otimes \sigma) v)$ for $v \in L'_{\mathbb{F}_{q^2}}$ by \cite[(2.7)]{GerWeil}. Hence $\Tr \rho_{\Sp(\omega'),\psi} (F)$ equals $\eta_0(-1)^n$ times the number of the fixed points of $1 \otimes \sigma$ on $L'_{\mathbb{F}_{q^2}}$. 
\end{proof}

\begin{lem}\label{lem:resHU}
The restriction $\rho_{\HSp(\omega'),\psi}|_{\HU(h_{2n}')}$ 
is isomorphic to 
$\rho_{\HU(h_{2n}'),\psi} \otimes (\nu \circ \det)$. 
\end{lem}
\begin{proof}
This follows from \cite[Theorem 3.3(a")]{GerWeil} and Lemma \ref{lem:HWcha}. 
\end{proof}
We regard the quadratic character $\oO(W,Q) \to \oO(W,Q)/\oU(W,h_1) \cong \bZ/2\bZ \hookrightarrow 
\overline{\mathbb{Q}}_{\ell}^{\times}$ 
as a character of $\Sp(V,\omega) \times \oO (W,Q)$ naturally, which we denote by $\kappa_0$. 
\begin{prop}\label{prop:compa}
We have an isomorphism 
$\rho_{\Sp(\omega'),\psi}|_{\Sp(V,\omega) \times \oO (W,Q)} 
\simeq \omega_{\mathrm{SpO},n,\psi} \otimes 
\kappa^{n(q-1)/2}_0$. 
\end{prop}
\begin{proof}
By Lemma \ref{lem:resHU}, 
the restriction $\rho_{\HSp(\omega'),\psi}|_{\HSpU(h_{2n}')}$ 
is isomorphic to 
$\pi_{h_{2n}',\psi}|_{\HSpU(h_{2n}')}$. 
Hence there exists a character 
$\eta \in (\mathbb{Z}/2\mathbb{Z})^{\vee}$ such that 
\[
 \rho_{\HSp(\omega'),\psi}|_{\HSpU(h_{2n}') \rtimes (\bZ/2\bZ)} 
 \simeq \pi_{h_{2n}',\psi}|_{\HSpU(h_{2n}') \rtimes (\bZ/2\bZ)} 
 \otimes \eta 
\]
by Schur's lemma. 
By Lemma \ref{dim} \ref{enu:dimodd} 
and Lemma \ref{spd}, we have $\eta=\kappa^{n(q-1)/2}_0$. 
\end{proof}

\subsection{Howe correspondence}\label{cation}
Let $\Irr \oO (W,Q)$ be the set of isomorphism classes of 
irreducible representations of $\oO (W,Q)$ 
over $\overline{\mathbb{Q}}_{\ell}$. 
For $\sigma \in \Irr \oO (W,Q)$, 
we put 
\[
 \Theta_n (\sigma) =\Hom_{\oO (W,Q)} (\sigma,\omega_{\mathrm{SpO},n,\psi}) 
\]
as an $\Sp (V,\omega)$-representation. 
We call $\sigma \mapsto \Theta_n (\sigma)$ 
the Howe correspondence for 
$\Sp(V,\omega) \times \oO (W,Q)$. 
This is compatible with the usual Howe correspondence,  
which is defined for $p \neq 2$, up to an explicit sign as 
Proposition \ref{prop:compa}. 

We identify $\oO (W,Q)$ with 
$\oU (W,h_1) \rtimes (\mathbb{Z}/2\mathbb{Z})$
as before. 
We set 
$\mu^{\vee}=\{\chi_0 \in \mu_{q+1}^{\vee} \mid \chi_0^2=1\}$. 
For a pair 
$(\chi_0,\kappa ) \in \mu^{\vee} \times \{ \pm 1 \}$, 
the map 
$\sigma_{\chi_0,\kappa} \colon \oO (W,Q) \to 
 \overline{\mathbb{Q}}_{\ell}^{\times};\ 
 (x,m) \mapsto \chi_0(x) \kappa^m$ for $x \in \mu_{q+1}$
and $m \in \mathbb{Z}/2\mathbb{Z}$ 
is a character. 
For a character $\chi \in \mu_{q+1}^{\vee}\setminus \mu^{\vee}$, 
the two-dimensional representation 
$\sigma_{\chi}=\mathrm{Ind}_{\oU (W,h_1)}^{\oO (W,Q)} \chi$ 
is irreducible. Note that 
$\sigma_{\chi} \simeq \sigma_{\chi^{-1}}$ as 
$\oO (W,Q)$-representations. 
Any irreducible representation of $\oO (W,Q)$
is isomorphic to the one of these representations. 
Note that the orthogonal algebraic group defined by 
$W$ and $Q$ is not connected. 
We say that an irreducible representation 
of $\oO (W,Q)$ is unipotent if 
its restriction to $\oU (W,h_1)$ contains a unipotent representation 
(\cf \cite[3.A]{AMRCorHowe}). 
The unipotent representations in $\Irr \oO (W,Q)$ are 
$\sigma_{1,+}$ and $\sigma_{1,-}$. 

\begin{lem}\label{lem:schi}
We have isomorphisms 
\[
\begin{cases}
\Theta_n (\sigma_{\chi}) \simeq 
 \omega_{\mathrm{SpO},n,\psi}[\chi] & \textrm{for $\chi \in \mu_{q+1}^{\vee}\setminus \mu^{\vee}$}, \\
 \Theta_n(\sigma_{\chi_0,\kappa}) \simeq \omega_{\mathrm{SpO},n,\psi}[\chi_0]^{\kappa} & \textrm{for $(\chi_0,\kappa) \in \mu^{\vee}\times \{\pm 1 \}$}
\end{cases}
\]
as $\Sp(V,\omega)$-representations. 
\end{lem}
\begin{proof}
The first isormophism follows from Frobenius reciprocity and 
the second one is clear. 
\end{proof}
\begin{lem}\label{lem:inn}
We have 
\[
\langle \rho_{\HU(h_{2n}'),\psi},\rho_{\HU(h_{2n}'),\psi} \rangle_{\Sp(V,\omega)}
=
\begin{cases}
2q+1 & \textrm{if $n=1$}, \\
2(q+1) & \textrm{if $n \geq 2$}. 
\end{cases}
\]
\end{lem}
\begin{proof}
By \cite[(2) in the proof of Corollary 4.5]{GerWeil}, 
the number of 
$\Sp(V,\omega)$-orbits in $V_{\mathbb{F}_{q^2}}$ equals  
$\langle \rho_{\HU(h_{2n}'),\psi},\rho_{\HU(h_{2n}'),\psi} \rangle_{\Sp(V,\omega)}$ (\cf \cite[Theorem 4.5(a)]{GerWeil}). 
We have the orbit $\{0\}$. 
We say that an element of $V_{\mathbb{F}_{q^2}}$ is 
decomposable if it is written as $v \otimes a$ for some 
$v \in V$ and $a \in \bF_{q^2}$. 
If an element of $V_{\mathbb{F}_{q^2}}$ is not decomposable, 
we say that it is indecomposable. 
Let $\{e_1,e_2\}$ be the basis of $\mathbb{F}_{q^2}$. 
An element of $V_{\mathbb{F}_{q^2}}$ is written as $v_1 \otimes e_1+v_2 \otimes e_2$ with $v_i \in V$. 
The set of the orbits of non-zero decomposable elements is identified with 
$\mathbb{P}^1(\mathbb{F}_q)$. 
The set of the orbits of indecomposable elements is identified with 
\[
\begin{cases}
 \mathbb{F}_q^{\times} & \textrm{if $n=1$}, \\
 \mathbb{F}_q & \textrm{if $n \geq 2$} 
\end{cases}
\]
by 
$\Sp(V,\omega)(v_1 \otimes e_1+v_2 \otimes e_2) 
 \mapsto \omega (v_1,v_2)$. 
Hence, the required assertion follows. 
\end{proof}

\begin{prop}\label{prop:irr}
The representations 
\[
 \Theta_n (\sigma) \quad \textrm{for $\sigma \in \Irr \oO (W,Q)$} 
\]
are irreducible and distinct as 
$\Sp(V,\omega)$-representations, 
except that $\Theta_1 (\sigma_{1,-})$. 
\end{prop}
\begin{proof}
We use the notation in Lemma \ref{lem:schi}. 
Let $\{\pm1\}$ act on $\mu_{q+1}^{\vee}\setminus \mu^{\vee}$ by $-1 \colon \chi \mapsto \chi^{-1}$. 
We have an isomorphism 
\begin{equation}\label{omegadec}
\omega_{\mathrm{SpO},n,\psi} \simeq 
\bigoplus_{\chi \in (\mu_{q+1}^{\vee}\setminus \mu^{\vee})/\{\pm 1\}} 
(\sigma_{\chi} \boxtimes \Theta_n(\sigma_{\chi}))
\oplus \bigoplus_{(\chi_0,\kappa) \in \mu^{\vee}\times \{\pm 1\}} 
(\sigma_{\chi_0,\kappa} \boxtimes \Theta_n(\sigma_{\chi_0,\kappa})) 
\end{equation}
as $\oO (W,Q) \times \Sp(V,\omega)$-representations. 
Recall $\omega_{\mathrm{SpO},n,\psi}|_{\mathrm{HU}(h'_{2n})} \simeq 
\rho_{\mathrm{HU}(h'_{2n}),\psi}$ by Lemma \ref{lem:F22}.
By Lemma \ref{dim}, Lemma \ref{lem:schi}, Lemma \ref{lem:inn} and \eqref{omegadec}, the claim follows. 
\end{proof}

\begin{cor}
The $\Sp(V,\omega)$-representation 
$\omega_{\mathrm{SpO},n,\psi}$ contains a unipotent cuspidal 
representation if and only if $n=2$. 
In this case, $\omega_{\mathrm{SpO},2,\psi}[1_{\mu_{q+1}}]^{-}$ is unipotent cuspidal.  
\end{cor}
\begin{proof}
By \cite[8.11. Remarks (1)]{LusIrrcl}, $\Sp(V,\omega)$ admits a unipotent cuspidal representation if and only if $n=s(s+1)$ with some integer $s \geq 1$.  
Every $\Sp(V,\omega)$-subrepresentation of $\omega_{\mathrm{SpO},n,\psi}$ in Lemma 
\ref{dim} is irreducible by Lemma \ref{lem:schi} and Proposition \ref{prop:irr}. 
By \cite[(8.11.1)]{LusIrrcl} and Lemma \ref{dim}, 
if $s \geq 2$, the $p$-adic valuation of the  dimension of the  unipotent cuspidal representation is greater than  
the dimension of any irreducible representation in $\omega_{\mathrm{SpO},n,\psi}$. Hence $\omega_{\mathrm{SpO},n,\psi}$ does not contain any unipotent cuspidal representation if $s \geq 2$. 
Therefore, for $\omega_{\mathrm{SpO},n,\psi}$
to contain a unipotent cuspidal representation, we must have $n=2$. 
Conversely, we show that $\omega_{\mathrm{SpO},2,\psi}[1_{\mu_{q+1}}]^-$ is unipotent cuspidal. 
This is irreducible by Proposition \ref{prop:irr} and of dimension $q(q-1)^2/2$ by Lemma \ref{dim}. This is isomorphic to $\theta_5$ in \cite[Table IV-2 in Appendix]{EnoChSp4} if $q$ is even 
and $\theta_{10}$ in \cite[\S 8]{SriChSp} if $q$ is odd, which is 
the unique irreducible representation of 
dimension $q(q-1)^2/2$. 
These are unipotent cuspidal by \cite[Theorem 8.2 and (8.11.1)]{LusIrrcl}. 
\end{proof}

\subsection{Relation with Lusztig induction}\label{proof}
We set 
\[
 J=\begin{pmatrix}
 \mathbf{0} & E_{n} \\
 -E_{n} & \mathbf{0}
 \end{pmatrix} \in \GL_{2n}(\bF_q), \quad 
 G=\Sp_{2n}=\left\{g \in \GL_{2n} \mid 
 {}^{\mathrm{t}}\! g J g=J \right\}. 
\]
Recall that 
\begin{equation}\label{abcd}
\begin{pmatrix}
 A & B \\
 C & D
\end{pmatrix} \in G \iff {}^{\mathrm{t}}\! A D-{}^{\mathrm{t}}\! C B=E_n ,\quad {}^{\mathrm{t}}\! C A
={}^{\mathrm{t}}\! AC,\quad {}^{\mathrm{t}}\! DB={}^{\mathrm{t}}\! B D. 
\end{equation}
We consider the Frobenius endomorphism 
$F \colon G \to G;\ 
g=(x_{i,j}) \mapsto (x_{i,j}^q)$. 
Let $\mathbf{0}_{n-1,1}$ denote the zero $(n-1) \times 1$-matrix. 
We consider the parabolic subgroup
\[
P_0=\begin{pmatrix}
\begin{array}{@{\,}cc|cc@{\,}}
\ast & \ast & \ast & \ast \\
\mathbf{0}_{n-1,1} & \ast & \ast & \ast \\
\hline
0 & \ast & \ast & \ast \\
\mathbf{0}_{n-1,1} & \ast & \ast & \ast 
\end{array}
\end{pmatrix} \subset G. 
\]
By using \eqref{abcd}, 
we have 
\begin{equation}\label{p}
P_0=\begin{pmatrix}
\ast & \ast & \ast & \ast \\
\mathbf{0} & \ast & \ast & \ast \\
0 & \mathbf{0} & \ast & \mathbf{0} \\
\mathbf{0} & \ast & \ast & \ast 
\end{pmatrix}. 
\end{equation}
We consider the Levi subgroup 
\[
 M_0=\left\{\begin{pmatrix}
 t & \mathbf{0} & 0 & \mathbf{0} \\
 \mathbf{0} & A_1 & \mathbf{0} & B_1 \\
 0 & \mathbf{0} & t^{-1} & \mathbf{0} \\
 \mathbf{0} & C_1 & \mathbf{0} & D_1 
 \end{pmatrix}\in G\ \bigg| \ 
 t \in \mathbb{G}_{\mathrm{m}},\ 
 \begin{pmatrix}
 A_1 & B_1 \\
 C_1 & D_1
 \end{pmatrix} \in \Sp_{2n-2}
 \right\} \simeq \mathbb{G}_{\mathrm{m}} \times \Sp_{2n-2}
\]
of $P_0$ and the unipotent radical 
\begin{equation}\label{up}
U_{P_0}=
\left\{\begin{pmatrix}
1 & \mathbf{a} & b & \mathbf{c} \\
\mathbf{0} & \mathbf{1} & {}^{\mathrm{t}} \mathbf{c} & \mathbf{0} \\
0 & \mathbf{0} & 1 & \mathbf{0} \\
\mathbf{0} & \mathbf{0} & -{}^{\mathrm{t}} \mathbf{a} & \mathbf{1} 
\end{pmatrix}\right\} 
\end{equation}
of $P_0$. We have $\dim U_{P_0}=2n-1$. 
Let 
\begin{align*}
w=\begin{pmatrix}
0 & \mathbf{0} & 1 & \mathbf{0} \\
\mathbf{0} & \mathbf{1} & \mathbf{0} & \mathbf{0} \\
-1 & \mathbf{0} & 0 & \mathbf{0} \\
\mathbf{0} & \mathbf{0} & \mathbf{0} & \mathbf{1} 
\end{pmatrix} \in \Sp_{2n}(\bF_q ). 
\end{align*} 
We take $g_w \in G(\overline{\mathbb{F}}_q)$ such that 
$g_w^{-1} F(g_w)=w$ by Lang's theorem. 

We put $P=g_w P_0 g_w^{-1}$, $M=g_w M_0 g_w^{-1}$ and 
$M_0^w = \{ m \in M_0 \mid F(m)=w^{-1}mw \}$. 
We put 
\[
 \widetilde{Y}_{P_0}' = \{ g \in G \mid g^{-1} F(g) \in w U_{P_0} \}, \quad 
 Y_{P_0}' = \{ gP_0 \in G/P_0 \mid g^{-1} F(g) \in P_0 w P_0 \}. 
\]
Let $\pi_{P_0}' \colon \widetilde{Y}_{P_0}' \to Y_{P_0}'$ 
be the natural morphism. 
Then we have the isomorphisms 
\[
 \widetilde{Y}_P \to \widetilde{Y}_{P_0}';\ g \mapsto g g_w , \quad 
 Y_P \to Y_{P_0}';\ gP \mapsto g g_w P_0 
\]
which give the commutative diagram 
\[
 \xymatrix{
 \widetilde{Y}_P \ar[r]^{\sim}\ar[d]_{\pi_P} 
 & 
 \widetilde{Y}_{P_0}' \ar[d]^{\pi_{P_0}'} \\
 Y_P \ar[r]^{\sim} & Y_{P_0}' . 
 }
\]
Therefore 
$\pi_{P_0}'$ is an $M_0^w$-torsor by Lemma \ref{lem:Mtor}.

Let $L_0=\overline{\mathbb{F}}_q e_1 \subset V_{\overline{\mathbb{F}}_q}=\overline{\mathbb{F}}_q^{2n}$. 
Let 
$s \colon V_{\overline{\mathbb{F}}_q} \times V_{\overline{\mathbb{F}}_q} \to 
\overline{\mathbb{F}}_q;\ (v,v') \mapsto {}^{\mathrm{t}}\! v J v'$. 
We have the isomorphism 
$f \colon 
G/P_0 \xrightarrow{\sim} \mathbb{P}(V_{\overline{\mathbb{F}}_q});\ gP_0 \mapsto g L_0$. 
For $L \in \mathbb{P}(V_{\overline{\mathbb{F}}_q})$, 
we put 
\[
 L^{\perp}=\{v \in V_{\overline{\mathbb{F}}_q} \mid s(v,v')=0\ 
 \textrm{for all $v' \in L$}\}. 
\]
We have 
\[
L_0^{\perp}=\left\{\sum_{i=1}^{2n} a_i e_i \in V_{\overline{\mathbb{F}}_q} \mid a_{n+1}= 0\right\}. 
\]
Let $Y'=\{ L \in \mathbb{P}(V_{\overline{\mathbb{F}}_q}) \mid F(L) \nsubseteq 
L^{\perp}\}$. 
For $g \in G$, the condition $F(g L_0) \nsubseteq (g L_0)^{\perp}
=g L_0^{\perp}$ is equivalent to 
$g^{-1} F(g) L_0 \nsubseteq L_0^{\perp}$. 
The morphism 
\[
 \varphi_0' \colon Y_{P_0}' \to Y' ;\ g P_0 \mapsto g L_0 
\]
is well-defined since 
we have 
$g^{-1} F(g) L_0 =p_1 w p_2 L_0 = p_1 w L_0 
 \nsubseteq p_1 L_0^{\perp} = L_0^{\perp}$ 
where $g^{-1} F(g) L_0 =p_1 w p_2$ for $p_1, p_2 \in P_0$. 

\begin{lem}\label{16}
The morphism $\varphi_0' \colon Y_{P_0}' \to Y'$ is an isomorphism. 
\end{lem}
\begin{proof}
Let $X_0=\{g=(g_{i,j}) \in G \mid g_{n+1,1} \neq 0\}$. 
We show that $P_0 w P_0 =X_0$. 
We easily check $P_0 w P_0 \subset X_0$. 
Let $P_0$ act on $X_0$ by right multiplication. 
Then we have a natural map 
$\phi \colon P_0 wP_0 /P_0 \to X_0/P_0$. 
By $P_0 \cap wP_0w^{-1}=M_0$, we have the isomorphisms
\begin{align*}
\phi_1 &\colon U_{P_0} \xrightarrow{\sim} P_0/(P_0 \cap w P_0 w^{-1}) \simeq 
P_0 wP_0 /P_0 ;\ u \mapsto  uwP_0 ,\\
 \phi_2 &\colon X_0/P_0 \xrightarrow{\sim} \mathbb{A}^{2n-1};\ 
 (g_{i,j}) P_0 \mapsto \biggl( \frac{g_{i,1}}{g_{n+1,1}} \biggr)_{1 \leq i \leq 2n,\, i \neq n+1}. 
\end{align*}
We can check 
that the composite $\phi_2 \circ \phi \circ \phi_1$ is 
an isomorphism by \eqref{up}. 
Thus the claim follows. 

Let $L \in Y'$. 
We take $g \in G$ such that $L=g L_0$. By 
$g^{-1} F(g) L_0 \nsubseteq L_0^{\perp}$, 
we have 
$g^{-1} F(g) \in X_0=P_0 wP_0$. Hence, we obtain the required assertion. 
\end{proof}

Let $S_{2n}'$ be 
the projective smooth variety defined by 
$\sum_{i=1}^n (x_i^q y_i-x_i y_i^q)=0$ in 
$\mathbb{P}_{\mathbb{F}_q}^{2n-1}$. 
Let $Y_{2n}'=\mathbb{P}_{\mathbb{F}_q}^{2n-1} \setminus S_{2n}'$. 
The canonical isomorphism 
$\bP (V_{\overline{\mathbb{F}}_q}) \simeq 
 \bP^{2n-1}_{\overline{\mathbb{F}}_q}$ 
induces an isomorphism 
$Y' \simeq Y_{2n,\overline{\mathbb{F}}_q}'$. 
Let $\varphi' \colon Y_{P_0}' \to Y_{2n,\overline{\mathbb{F}}_q}'$ 
be the composite of 
$\varphi_0' \colon Y_{P_0}' \to Y'$ and 
the isomorphism 
$Y' \simeq Y_{2n,\overline{\mathbb{F}}_q}'$. 

Let $\mathbf{x}=(x_1\ \cdots\ x_n)$,
 $\mathbf{y}=(y_1\ \cdots\ y_n)$, 
 $\mathbf{z}=(z_1\ \cdots\ z_n)$ and 
 $\mathbf{w}=(w_1\ \cdots\ w_n)$. 
We write as 
\[
g=
\begin{pmatrix}
\begin{array}{@{\,}c|c@{\,}}
A & B \\
\hline
C & D
\end{array}
\end{pmatrix}
=
\begin{pmatrix}
\begin{array}{@{\,}cc|cc@{\,}}
{}^{\mathrm{t}} \mathbf{x} & \ast & {}^{\mathrm{t}} \mathbf{z} & \ast \\
\hline 
{}^{\mathrm{t}} \mathbf{y} & \ast & {}^{\mathrm{t}} \mathbf{w} & \ast 
\end{array}
\end{pmatrix} \in \widetilde{Y}_{P_0}'. 
\]
We have 
\[
w U_{P_0}=\left\{
\begin{pmatrix}
0 & \mathbf{0} & 1 & \mathbf{0} \\
\mathbf{0} & \mathbf{1} & {}^{\mathrm{t}} \mathbf{c} & \mathbf{0} \\
-1 & \mathbf{a} & -b & -\mathbf{c} \\
\mathbf{0} & \mathbf{0} & -{}^{\mathrm{t}} \mathbf{a} & 1
\end{pmatrix}\right\}. 
\]
By $g^{-1} F(g) \in w U_{P_0}$, we have 
$F(\mathbf{x})=-\mathbf{z}$ and 
$F(\mathbf{y})=-\mathbf{w}$. 
By ${}^{\mathrm{t}}\! AD-{}^{\mathrm{t}}\! C B=E_n$, 
we have $\mathbf{y} {}^{\mathrm{t}}\! F(\mathbf{x}) -\mathbf{x} {}^{\mathrm{t}}\! F(\mathbf{y})=1$. 
Let $\widetilde{Y}_{2n}'$ be the affine smooth variety 
over $\mathbb{F}_q$ defined by $\sum_{i=1}^n(x_i^qy_i-x_iy_i^q)=1$
in $\mathbb{A}_{\mathbb{F}_q}^{2n}$.
We have the morphisms 
\begin{align*}
 \widetilde{\pi}' & \colon \widetilde{Y}_{P_0}' 
 \to \widetilde{Y}_{2n,\overline{\mathbb{F}}_q}' ;\ 
 g \mapsto (\mathbf{x},\mathbf{y}), \\ 
 f_{Y'_{2n}} & \colon 
 \widetilde{Y}_{2n,\overline{\mathbb{F}}_q}' 
 \to Y_{2n,\overline{\mathbb{F}}_q}';\ 
 (\mathbf{x},\mathbf{y}) \mapsto
 [x_1:\cdots:x_n:y_1:\cdots:y_n]. 
\end{align*}
We identify $\Sp_{2n-2}(\bF_q)$ with the subgroup $\{1\} \times \Sp_{2n-2}(\bF_q) \subset M_0^w$. 
Then $\widetilde{\pi}'$ induces a morphism 
\[
 \widetilde{\varphi}' \colon 
 \widetilde{Y}_{P_0}'/\Sp_{2n-2}(\bF_q) \to 
 \widetilde{Y}_{2n,\overline{\mathbb{F}}_q}'. 
\]
We have the commutative diagram 
\[
 \xymatrix{
 \widetilde{Y}_{P_0}' \ar[d] \ar[rd]^-{\widetilde{\pi}'} & \\ 
 \widetilde{Y}_{P_0}'/\Sp_{2n-2}(\bF_q) \ar[r]_-{\widetilde{\varphi}'} 
 \ar[d]_{f_{Y_{P_0}'}} 
 & \widetilde{Y}'_{2n,\overline{\mathbb{F}}_{q}} \ar[d]^{f_{Y'_{2n}}} \\ 
 Y_{P_0}' \ar[r]_-{\varphi'} & Y'_{2n,\overline{\mathbb{F}}_{q}} 
}
\]
where the vertical morphisms are natural ones. 
Since 
$f_{Y'_{P_0}}$ and $f_{Y'_{2n}}$ are $\mu_{q+1}$-torsors 
and $\varphi'$ is an isomorphism, 
$\widetilde{\varphi}'$ is also an isomorphism. 

\begin{rem}\label{rem:ssec2n}
The results in Subsection \ref{ssec:Grel} 
hold even if we replace 
$n$, 
$\pi \colon \mathbb{A}_{\mathbb{F}_q}^n 
\to \mathbb{A}_{\mathbb{F}_q}^1$, 
$Y_{n}$ etc. 
by 
$2n$, 
$\pi' \colon \mathbb{A}_{\mathbb{F}_q}^{2n} 
\to \mathbb{A}_{\mathbb{F}_q}^1$, 
$Y'_{2n}$ etc. 
\end{rem}

\begin{prop}\label{lastb}
Let $\chi \in \mu_{q+1}^{\vee}$. 
We have  
\begin{align*}
(H_{\rm c}^{2n}(X_{2n,\overline{\mathbb{F}}_q}',\overline{\mathbb{Q}}_{\ell})[\psi])[\chi]
&=
\begin{cases}
\Theta_n(\sigma_{\chi,+}) \oplus \Theta_n(\sigma_{\chi,-}) 
& \textrm{for $\chi \in \mu^{\vee}$}, \\
\Theta_n(\sigma_{\chi}) 
& \textrm{for $\chi \in \mu_{q+1}^{\vee} \setminus \mu^{\vee}$}
\end{cases} 
 \\
&=
\begin{cases}
-R^G_{M \subset P}(\chi \otimes 1_{\Sp_{2n-2}(\bF_q)}) & \textrm{if  
$\chi \in \mu_{q+1}^{\vee} \setminus \{1\}$}, \\
1_{G^F}-R^G_{M \subset P}(1_{M^F}) 
& \textrm{if $\chi =1$}
\end{cases}
\end{align*}
as characters of $G^F$. 
\end{prop}
\begin{proof}
The first equality follows from Lemma \ref{lem:schi}. 
The equality between the first one and the third one follows from Remark \ref{rem:ssec2n} 
(\cf Proposition \ref{cohXY}) and \eqref{ttp}, 
since $\widetilde{\varphi}'$ is an isomorphism. 
\end{proof}

\begin{cor}\label{cc2}  
Let $\sigma \in \Irr \oO (W,Q)$. 
Assume that 
$\sigma \neq \sigma_{1,-}$
if $n=1$. 
Then the $\Sp(V,\omega)$-representation  
$\Theta_n (\sigma)$ 
is unipotent if and only if $\sigma$ is unipotent. 
\end{cor}
\begin{proof}
The claim follows from \eqref{vv} and Proposition \ref{lastb}. 
\end{proof}

\begin{rem}
It is also possible to show 
Corollary \ref{cc2} using results in 
\cite{GuTiCross} if $p=2$, and 
results in \cite{TiZaMinCh} and \cite{NguLdim} 
if $p \neq 2$. 
The proof in this paper is geometric and 
does not depend on the parity of $p$. 
\end{rem}

\begin{rem}
We assume that $q$ is odd. 
If $q$ is large enough 
as in \cite[(3.11)(2)]{AdMoUnip}, 
then Corollary \ref{cc2} follows 
from \cite[Theorem 3.5(2)]{AdMoUnip}.
\end{rem}

\begin{rem}\label{lastbb}
We use the same notation as Remark \ref{remuni}. 
We have 
$\epsilon_G=(-1)^n$ and 
$\epsilon_{M}=(-1)^{n-1}$. As in Remark \ref{remuni}, we have 
\[
\dim R_{M \subset P}^G(\chi \otimes 1_{\Sp_{2n-2}(\bF_q)})=\epsilon_G
\epsilon_{M} \frac{|G^F|_{p'}}{|M^F|_{p'}}
=-\frac{q^{2n}-1}{q+1}  
\]
for any $\chi \in \mu_{q+1}^{\vee}$. 
This is compatible with Lemma \ref{dim} and Proposition \ref{lastb}.  
\end{rem}

\noindent
Naoki Imai\\
Graduate School of Mathematical Sciences, The University of Tokyo, 3-8-1 Komaba, Meguro-ku, Tokyo, 153-8914, Japan \\
naoki@ms.u-tokyo.ac.jp \\[0.5cm]
Takahiro Tsushima\\ 
Department of Mathematics and Informatics, 
Faculty of Science, Chiba University
1-33 Yayoi-cho, Inage, 
Chiba, 263-8522, Japan \\
tsushima@math.s.chiba-u.ac.jp

\begin{thebibliography}{AMR96}
\providecommand{\url}[1]{\texttt{#1}}
\providecommand{\urlprefix}{URL }
\providecommand{\eprint}[2][]{\url{#2}}

\bibitem[AM93]{AdMoUnip}
J.~Adams and A.~Moy, Unipotent representations and reductive dual pairs over
  finite fields, Trans. Amer. Math. Soc. 340 (1993), no.~1, 309--321.

\bibitem[AMR96]{AMRCorHowe}
A.-M. Aubert, J.~Michel and R.~Rouquier, Correspondance de {H}owe pour les
  groupes r\'{e}ductifs sur les corps finis, Duke Math. J. 83 (1996), no.~2,
  353--397.

\bibitem[BR06]{BoRoCoxmod}
C.~Bonnaf\'{e} and R.~Rouquier, Coxeter orbits and modular representations,
  Nagoya Math. J. 183 (2006), 1--34.

\bibitem[Del77]{DelCoet}
P.~Deligne, Cohomologie \'etale, Lecture Notes in Mathematics, Vol. 569,
  Springer-Verlag, Berlin-New York, 1977, s{\'e}minaire de G{\'e}om{\'e}trie
  Alg{\'e}brique du Bois-Marie SGA 4${\frac{1}{2}}$, Avec la collaboration de
  J. F. Boutot, A. Grothendieck, L. Illusie et J. L. Verdier.

\bibitem[Des08]{DesExcSp}
T.~Deshpande, An exceptional representation of $\mathit{Sp}(4,\mathbb{F}_q)$,
  2008, arXiv:0804.2722.

\bibitem[DL76]{DeLuRep}
P.~Deligne and G.~Lusztig, Representations of reductive groups over finite
  fields, Ann. of Math. (2) 103 (1976), no.~1, 103--161.

\bibitem[DM91]{DiMiRepLie}
F.~Digne and J.~Michel, Representations of finite groups of {L}ie type, vol.~21
  of London Mathematical Society Student Texts, Cambridge University Press,
  Cambridge, 1991.

\bibitem[DM14]{DiMiParDL}
F.~Digne and J.~Michel, Parabolic {D}eligne-{L}usztig varieties, Adv. Math. 257
  (2014), 136--218.

\bibitem[Dud09]{DudDLGG}
O.~Dudas, Deligne-{L}usztig restriction of a {G}elfand-{G}raev module, Ann.
  Sci. \'{E}c. Norm. Sup\'{e}r. (4) 42 (2009), no.~4, 653--674.

\bibitem[Enn63]{EnnChar}
V.~Ennola, On the characters of the finite unitary groups, Ann. Acad. Sci.
  Fenn. Ser. A I No. 323 (1963), 35.

\bibitem[Eno72]{EnoChSp4}
H.~Enomoto, The characters of the finite symplectic group {${\rm Sp}(4,\,q)$},
  {$q=2^{f}$}, Osaka J. Math. 9 (1972), 75--94.

\bibitem[G{\'{e}}r77]{GerWeil}
P.~G{\'{e}}rardin, Weil representations associated to finite fields, J. Algebra
  46 (1977), no.~1, 54--101.

\bibitem[GH07]{GuHaGeoW}
S.~Gurevich and R.~Hadani, The geometric {W}eil representation, Selecta Math.
  (N.S.) 13 (2007), no.~3, 465--481.

\bibitem[GT04]{GuTiCross}
R.~M. Guralnick and P.~H. Tiep, Cross characteristic representations of even
  characteristic symplectic groups, Trans. Amer. Math. Soc. 356 (2004), no.~12,
  4969--5023.

\bibitem[HM78]{HoMaTT}
R.~Hotta and K.~Matsui, On a lemma of {T}ate-{T}hompson, Hiroshima Math. J. 8
  (1978), no.~2, 255--268.

\bibitem[How73]{HoweCharW}
R.~E. Howe, On the character of {W}eil's representation, Trans. Amer. Math.
  Soc. 177 (1973), 287--298.

\bibitem[How79]{Howthetainv}
R.~Howe, {$\theta $}-series and invariant theory, in Automorphic forms,
  representations and {$L$}-functions ({P}roc. {S}ympos. {P}ure {M}ath.,
  {O}regon {S}tate {U}niv., {C}orvallis, {O}re., 1977), {P}art 1, Proc. Sympos.
  Pure Math., XXXIII, Amer. Math. Soc., Providence, R.I., 1979 pp. 275--285.

\bibitem[IT15]{ITlgsw1}
N.~Imai and T.~Tsushima, Local {G}alois representations of {S}wan conductor
  one, 2015, arXiv:1509.02960.

\bibitem[IT17]{ITstab3}
N.~Imai and T.~Tsushima, Stable models of {L}ubin-{T}ate curves with level
  three, Nagoya Math. J. 225 (2017), 100--151.

\bibitem[IT19]{ITShinlift}
N.~Imai and T.~Tsushima, Shintani lifts for Weil representations of unitary
  groups over finite fields, 2019, arXiv:1906.03615.

\bibitem[IT20]{ITModlW}
N.~{Imai} and T.~{Tsushima}, {Mod $\ell$ Weil representations and
  Deligne--Lusztig inductions for unitary groups}, 2020, arXiv:2002.05872.

\bibitem[KL90]{KlLiSubcl}
P.~Kleidman and M.~Liebeck, The subgroup structure of the finite classical
  groups, vol. 129 of London Mathematical Society Lecture Note Series,
  Cambridge University Press, Cambridge, 1990.

\bibitem[Lau87]{LauTFcW}
G.~Laumon, Transformation de {F}ourier, constantes d'\'equations fonctionnelles
  et conjecture de {W}eil, Inst. Hautes \'Etudes Sci. Publ. Math.  (1987),
  no.~65, 131--210.

\bibitem[Leh74]{LehWuni}
G.~I. Lehrer, Weil representations and cusp forms on unitary groups, Bull.
  Amer. Math. Soc. 80 (1974), 1137--1141; corrigendum, ibid. 81 (1975), 636.

\bibitem[LS77]{LuSrChar}
G.~Lusztig and B.~Srinivasan, The characters of the finite unitary groups, J.
  Algebra 49 (1977), no.~1, 167--171.

\bibitem[Lus76]{LusFinunip}
G.~Lusztig, On the finiteness of the number of unipotent classes, Invent. Math.
  34 (1976), no.~3, 201--213.

\bibitem[Lus77]{LusIrrcl}
G.~Lusztig, Irreducible representations of finite classical groups, Invent.
  Math. 43 (1977), no.~2, 125--175.

\bibitem[Lus78]{LusRepChe}
G.~Lusztig, Representations of finite {C}hevalley groups, vol.~39 of CBMS
  Regional Conference Series in Mathematics, American Mathematical Society,
  Providence, R.I., 1978, expository lectures from the CBMS Regional Conference
  held at Madison, Wis., August 8--12, 1977.

\bibitem[Ngu10]{NguLdim}
H.~N. Nguyen, Low-dimensional complex characters of the symplectic and
  orthogonal groups, Comm. Algebra 38 (2010), no.~3, 1157--1197.

\bibitem[Sai72]{SaiRepsym}
M.~Saito, Repr\'{e}sentations unitaires des groupes symplectiques, J. Math.
  Soc. Japan 24 (1972), 232--251.

\bibitem[SGA77]{SGA5}
Cohomologie {$l$}-adique et fonctions {$L$}, Lecture Notes in Mathematics, Vol.
  589, Springer-Verlag, Berlin-New York, 1977, s{\'e}minaire de G{\'e}ometrie
  Alg{\'e}brique du Bois-Marie 1965--1966 (SGA 5), Edit{\'e} par Luc Illusie.

\bibitem[Sri68]{SriChSp}
B.~Srinivasan, The characters of the finite symplectic group {${\rm
  Sp}(4,\,q)$}, Trans. Amer. Math. Soc. 131 (1968), 488--525.

\bibitem[Sri79]{SriWcl}
B.~Srinivasan, Weil representations of finite classical groups, Invent. Math.
  51 (1979), no.~2, 143--153.

\bibitem[Tie97]{TieWglirr}
P.~H. Tiep, Weil representations as globally irreducible representations, Math.
  Nachr. 184 (1997), 313--327.

\bibitem[Tie10]{TieDualcl}
P.~H. Tiep, Dual pairs of finite classical groups in cross characteristics, in
  Character theory of finite groups, vol. 524 of Contemp. Math., pp. 161--179,
  Amer. Math. Soc., Providence, RI, 2010.

\bibitem[TZ96]{TiZaMinCh}
P.~H. Tiep and A.~E. Zalesskii, Minimal characters of the finite classical
  groups, Comm. Algebra 24 (1996), no.~6, 2093--2167.

\bibitem[Wei64]{Weiopuni}
A.~Weil, Sur certains groupes d'op\'{e}rateurs unitaires, Acta Math. 111
  (1964), 143--211.

\end{thebibliography}
\end{document}